\documentclass{article}
\usepackage{graphicx} % Required for inserting images
\usepackage[english]{babel}
\usepackage{amsthm}
\usepackage{amssymb}
\usepackage{fancyhdr}
\usepackage{pict2e} 
\newcommand\quotientthree[2]{{#1}\Big/{#2}}
\usepackage{url}
\usepackage{csquotes} % Recommended for biblatex with babel/polyglossia
\usepackage{amsmath}
\usepackage{geometry}
\usepackage{tikz}
\usepackage{tikz-cd}
\usepackage{subcaption}
\usepackage{tikz}
\usetikzlibrary{positioning}
\usepackage{graphicx}
\usetikzlibrary{shapes.geometric, positioning}
\usepackage{subcaption}
\usepackage{caption}
\usetikzlibrary{positioning,chains,fit,shapes,calc}
\date{}
\usepackage{tikz}

\title{On Posets of Classes of Automorphic Subgroups of Finite Groups }
\author{Sachin Ballal\footnote{Corresponding Author} \hspace{0.1cm}and Tushar Halder \vspace{0.2cm}\\ School of Mathematics and Statistics, University of Hyderabad, 500046, India \\ sachinballal@uohyd.ac.in, 24mmpp03@uohyd.ac.in }
\theoremstyle{definition}
\newtheorem{theorem}{Theorem}[section]
\newtheorem{corollary}{Corollary}[theorem]
\newtheorem{lemma}[theorem]{Lemma}
\newtheorem{remark}{Remark}
  \usepackage[
    backend=biber,
    style=numeric,
  ]{biblatex}

\providecommand{\keywords}[1]
{
  \small	
  \textbf{{Keywords:}} #1
}
\begin{document}
\maketitle
\begin{abstract}
     In \cite{tarnauceanu2015posetsclassesisomorphicsubgroups}, Tărnăuceanu studied the poset $\text{Iso}(G)$, of isomorphic classes of subgroups of a finite group $G$ and proposed several questions for further research. In this paper, we study the poset $\text{AutCl}(G)$, of classes of automorphic subgroups of finite group $G$. We introduce a partial order on $\text{AutCl}(G)$ to tackle problem 5 mentioned in $\S4$ of \cite{tarnauceanu2015posetsclassesisomorphicsubgroups}. More precisely, we prove that $\text{AutCl}(D_n)$ and $\text{AutCl}(Q_{4m})$ are distributive lattices. Moreover, we characterize all classes of finite groups for which $\text{AutCl}(G)$ is a chain. \vspace{0.2cm}\\
    
\end{abstract}
\section{Introduction}
The theory of subgroup lattices began with Ada Rottländer \cite{Ada} and this study was motivated by some questions arising from field extensions. It is well known that the set of all subgroups of a group forms a lattice, where meet is the intersection of subgroups and join is the subgroup generated by union of subgroups. The study of structure of groups using the lattice of subgroups is a prominent way which is explored by many researchers viz., Iwasawa \cite{iwasawa2001collected}, Schmidt \cite{Schmidt+1994}, Suzuki \cite{Suzuki1956}, etc. \par
In \cite{tarnauceanu2015posetsclassesisomorphicsubgroups}, Tărnăuceanu introduced the poset $\text{Iso}(G)$, which is defined as the set of classes of isomorphic subgroups of $G$ and studied its properties.
For a positive integer $n$, the \emph{dihedral group} of order
$2n$, denoted by $D_n$, is defined as
\begin{equation*}
    D_n=\bigl<r,s\hspace{0.2cm}|\hspace{0.2cm} r^n=e,\hspace{0.2cm}s^2=e,\hspace{0.2cm}srs^{-1}=r^{-1}\bigr>.
\end{equation*} 
In the following theorem, a complete listing of subgroups of $D_n$ is given.
\begin{theorem}
     \cite{conrad} Every subgroup of $D_n$ is cyclic or dihedral. A complete listing of the subgroups is as follows: 
\begin{enumerate}
    \item $\bigl<r^d\bigr>$, where $d|n$, with index $2d$,
    \item $\bigl<r^d,r^is\bigr>$, where $d|n$ and $0\le i\le d-1$, with index $d$.
\end{enumerate}
Every subgroup of $D_n$ occurs exactly once in this listing.
\end{theorem}
\begin{remark}
\begin{enumerate}
    \item A subgroup of $D_n$ is said to be of \textbf{Type (1)} if it is cyclic as stated in (1) of Theorem 1.1.
    \item A subgroup of $D_n$ is said to be of \textbf{Type (2)} if it is dihedral subgroup as stated in (2) of Theorem 1.1.
\end{enumerate}
\end{remark}
For $m\ge 2$, the \emph{more generalized quaternion group} of order $4m$, denoted by $Q_{4m}$, is defined as:
\begin{equation*}
    Q_{4m}=\bigl<x,y \hspace{0.2cm}|\hspace{0.2cm} x^{2m}=e=y^4, \hspace{0.2cm} yxy^{-1}=x^{-1}, \hspace{0.2cm} x^m=y^2 \bigr>.
\end{equation*}
Note that for $m=1,2$, we have $Q_4\cong \mathbb{Z}_4$ and $Q_8$ is the usual quaternion group with 8 elements. 
In the following theorem, a complete listing of subgroups of $Q_{4m}$ is given.
\begin{theorem} \cite{conradquad}
    For $m\ge 1$, every subgroup of $Q_{4m}$ is cyclic or dicyclic. A complete listing of subgroups of $Q_{4m}$ is as follows:
    \begin{enumerate}
        \item $\bigl<x^d\bigr>$, where $d|2m$, with index $2d$,
        \item $\bigl<x^d,x^iy\bigr>$, where $d|m$ and $0\le i \le d-1$, with index $d$.
    \end{enumerate}
    Every subgroup of $Q_{4m}$ occurs exactly once in this listing.
\end{theorem}
 We denote the lattice of subgroups of a group $G$ by $L(G)$ and the identity element of $G$ by $e$. The \emph{exponent} of a finite group $G$ is the smallest positive integer $n$, such that for all $g\in G$, $g^n=e$. We denote the \emph{chain} with $n$ elements by $C_n$. The lattice $M_2$ stands for the lattice with 4 elements as shown in Figure 1.
 \begin{figure}[h]
    \centering
    \begin{tikzpicture}[scale=1.5,  every node/.style={circle, fill=black, minimum size=6pt, inner sep=0pt}]
				
				% Nodes in set U
				\node[label=left:] (u1) at (-0.5,0) {};
				\node[label=left:] (u3) at (0.5,0) {};
				\node[label=left:] (u4) at (0,0.5) {};
				\node[label=left:] (u5) at (0,-0.5) {};
            
				% Edges
				\draw (u1) -- (u4);
				\draw (u4) -- (u3);
				\draw (u3) -- (u5);
				\draw (u5) -- (u1);

		\end{tikzpicture}
        \caption{$M_2$}
\end{figure}\\
 We denote the lattice of positive divisors of an integer $n$ by $T(n)$. A \emph{homomorphism} $\phi$ from group $(G_1, *_1)$ to group $(G_2,*_2)$ is a map such that $\phi(x*_1y)=\phi(x)*_2\phi(y)$, for all $x,y \in G_1$. An \emph{automorphism} of a group is a bijective homomorphism from the group to itself. The set of all automorphisms of a group $G$ forms a group and is denoted by $\text{Aut}(G)$. A finite lattice $L$ with the smallest element 0 and the largest element 1 is said to be \emph{complemented} if for $a\in L$, there is $b\in L$ such that $a\wedge b=0$ and $a\vee b=1$. For a finite cyclic group $\bigl<a\bigr>$, the order of $a^k$ is $\frac{|a|}{\text{gcd}(|a|,k)}$. For an odd prime $p$, upto isomorphism, there is a unique non abelian group of order $p^3$ with exponent $p$. This group is isomorphic to the group of all upper unitriangular $3\times 3$ matrices over $\mathbb{Z}_p$ and is denoted by $\text{Heis}(\mathbb{Z}_p)$.\begin{equation*}
    \text{Heis}(\mathbb{Z}_p)\cong \Biggl\{\begin{pmatrix}
        1 & x_1 &x_2 \\ 0 & 1 & x_3 \\ 0 & 0 &1 
    \end{pmatrix}  \hspace{0.2cm}|\hspace{0.2cm} x_1,x_2,x_3 \in \mathbb{Z}_p \Biggr\}
\end{equation*} Note that $\text{Heis}(\mathbb{Z}_p)$ is also isomorphic to $ \bigl<x,y,z \hspace{0.2cm}|\hspace{0.2cm} xyx^{-1}y^{-1}=z, \hspace{0.2cm}xz=zx, \hspace{0.2cm} yz=zy, \hspace{0.2cm}x^p=y^p=z^p=e \bigr>$.   \par
    In $\S 2$, we study the poset, $\text{AutCl}(G)$ of classes of automorphic subgroups of a finite group $G$, in particular for dihedral group $D_n$ and more generalized quaternion group $Q_{4m}$. We prove that the poset of automorphic class of subgroups of $D_n$ and $Q_{4m}$ form distributive lattices. In $\S3$, we characterize finite groups $G$, for which $\text{AutCl}(G)$ is a chain. This characterization turns out to be very similar to that of $\text{Iso}(G)$ as described in \cite{tarnauceanu2015posetsclassesisomorphicsubgroups}. Lastly in $\S4$, we raise some questions regarding $\text{AutCl}(G)$. \par
   For more details on lattices, groups and subgroup lattices, one may refer (\cite{birkhoff1940lattice},\cite{Gratzer},\cite{davey2002introduction}),  (\cite{robinson1996course},\cite{suzuki1982group},\cite{suzuki1986group}) and (\cite{Schmidt+1994},\cite{Suzuki1956},\cite{bookTarn}), respectively.
   
\section{The Automorphic Classes of Subgroups $\text{AutCl}(\textbf{\emph{G}})$}
In this section, we study $\text{AutCl}(G)$, the automorphic classes of subgroups of a finite group $G$. We show that $\text{AutCl}(D_n)$ and $\text{AutCl}(Q_{4m})$ are distributive lattices. \par 
In \cite{tarnauceanu2015posetsclassesisomorphicsubgroups}, Tărnăuceanu proposed a problem to study the classes of subgroups of a finite group $G$ with respect to the equivalence relation $\equiv$ on $L(G)$, where $\equiv$ is defined as follows: 
\begin{equation*}
    H\equiv K \hspace{0.1cm}\text{if and only if}\hspace{0.1cm} \text{there is} \hspace{0.1cm}f \in \text{Aut}(G) \hspace{0.1cm}\text{such that} \hspace{0.1cm}f(H) = K. 
\end{equation*}
\begin{lemma}
Let $G$ be a finite group. Define a relation $\lesssim$ on the set of equivalence classes ($\quotientthree{L(G)}{\equiv}$), of subgroups of $G$, as follows: 
\begin{equation}
    [H]\lesssim [K] \hspace{0.1cm} \text{if and only if } \text{there are}\hspace{0.1cm} H_1 \in [H],\hspace{0.1cm} K_1 \in [K] \hspace{0.1cm}\text{and } f \in \text{Aut}(G) \hspace{0.1cm}\text{such that } f(H_1) \subseteq K_1.
\end{equation}
Then $\biggl(\quotientthree{L(G)}{\equiv}, \lesssim \bigg)$ is a partially ordered set.     
\end{lemma}
\begin{remark}
    The relation defined in (1) is independent of the choice of representative. If $[H]\lesssim [K]$, then there are $H_1\in [H]$, $K_1\in [K]$ and $f\in \text{Aut}(G)$ with $f(H_1)\subseteq K_1$. Moreover, by the definition of $\equiv$, there are $\phi_1, \phi_2 \in \text{Aut}(G)$ with $\phi_1(H)=H_1$ and $\phi_2(K_1)=K$, consequently, $\phi_2\circ f\circ\phi_1 (H)\subseteq K$. 
\end{remark}

\begin{proof}
    In the light of Remark 2, reflexivity of $\lesssim$ follows immediately, as identity automorphism maps any subgroup of $G$ to itself. \par 
For antisymmetry, let $[H_1],[H_2] \in (\quotientthree{L(G)}{\equiv})$ be such that $[H_1]\lesssim [H_2]$ and $[H_2]\lesssim [H_1]$. So, there are $f_1, f_2 \in \text{Aut}(G)$ with $f_1(H_1)\subseteq H_2$ and $f_2(H_2)\subseteq H_1$. As, $|H_1|=|f_1(H_1)|\le |H_2|=|f_2(H_2)|\le |H_1|$, so $|f_1(H_1)|=|H_2|$, which implies $f_1(H_1)=H_2$ and consequently, $[H_1]=[H_2]$. \par
Now, for transitivity, let $[H_1]\lesssim[H_2]$ and $[H_2]\lesssim [H_3]$, then there are maps $f_1,f_2\in \text{Aut}(G)$ with $f_1(H_1)\subseteq H_2$ and $f_2(H_2)\subseteq H_3$, so, $f_2\circ f_1(H_1)\subseteq H_3$ and consequently, $[H_1]\lesssim [H_3]$. 
\end{proof}
 Henceforth, we will call the partially ordered set $\biggl ( \quotientthree{L(G)}{\equiv}, \lesssim \biggr)$ as \emph{the poset of automorphic classes of subgroups of finite group} $G$ and will denote it by $\text{AutCl}(G)$. \vspace{0.2cm}\\
 \textbf{Examples:} \begin{enumerate}
     \item For a natural number $n$, consider the cyclic group $\mathbb{Z}_n$. The map $\phi: L(\mathbb{Z}_n) \to \text{AutCl}(\mathbb{Z}_n),$ defined by $H \longmapsto [H]$, is a join and meet isomorphism between the lattice $L(\mathbb{Z}_n)$ and the poset $\text{AutCl}(\mathbb{Z}_n)$ and hence, $\text{AutCl}(\mathbb{Z}_n)$ is a lattice.
     \item $\text{AutCl}(\mathbb{Z}_2 \times \mathbb{Z}_2)\cong C_3$ and $\text{AutCl}(Q_8)\cong C_4$, where $Q_8$ is the quaternion group with 8 elements.
 \end{enumerate}
 In the next result, we show that, if $[H]\in \text{AutCl}(G)$ has a complement, then $H\in L(G)$ has a complement. 
\begin{theorem}
    Let $G$ be a finite group. If $\text{AutCl}(G)$ is complemented lattice, then $L(G)$ is also complemented. More precisely, let $(\text{AutCl(G)}, \wedge',\vee')$  be a lattice such that for some $[H]\in \text{AutCl(G)}$, there exists $[K]\in \text{AutCl(G)}$ with \begin{equation*}
    [H]\wedge ' [K]=[\{e\}] \hspace{0.2cm} \text{and} \hspace{0.2cm} [H]\vee' [K]=[G],
\end{equation*}  
then \begin{equation*}
    H\wedge K= \{e\} \hspace{0.2cm} \text{and} \hspace{0.2cm} H \vee K=G \hspace{0.2cm}\text{in } (L(G),\wedge,\vee).
\end{equation*}
\end{theorem}
  \begin{proof} For a finite group $G$, assume that $(\text{AutCl}(G), \wedge',\vee')$ is a complemented lattice. Moreover, it is well known that $(L(G), \wedge,\vee)$ is a lattice, where $H\wedge K=H\cap K$ and $H\vee K=\bigl<H\cup K\bigr>$. As, $\text{AutCl}(G)$ is complemented, we have, for any $[H]\in \text{AutCl}(G)$, there exists $[K]\in \text{AutCl}(G)$ such that \begin{equation*}
    [H]\wedge'[K]=[\{e\}]\hspace{0.2cm} \text{and }\hspace{0.2cm} [H]\vee'[K]=[G].
\end{equation*}
Clearly, $H\wedge K\leq H,K$ in $L(G)$. Moreover, as $\text{id}_G(H\wedge K) \subseteq H, K $, where $\text{id}_G\in \text{Aut}(G)$ is the identity automorphism of $G$, so, by definition of $\text{AutCl}(G)$, $[H\wedge K]\lesssim [H],[K]$, which implies $[H\wedge K]\lesssim [H]\wedge'[K]=[\{e\}]$ and hence $H\wedge K=\{e\}$. Furthermore, $H,K\leq H\vee K$, this implies $[H],[K]\lesssim[H\vee K]$, as $\text{id}_G(H),\text{id}_G(K)\subseteq H\vee K$. Thus, $[G]=[H]\vee'[K]\lesssim[H\vee K]$, and consequently, $H\vee K=G$. \end{proof}
\begin{remark}
    The converse of Theorem 2.2 need not be true. For instance, $L(K_4)$ is complemented but $\text{AutCl}(K_4)$ is not.
    \end{remark}
The following result is of great interest.
\begin{theorem}
    \cite{conrad} For $n\ge 3$, \begin{equation*}
    \text{Aut}(D_n) \cong \Biggl\{\begin{pmatrix}
    a & b \\0 & 1
    \end{pmatrix}
    \hspace{0.2cm}|\hspace{0.2cm} a \in \mathbb{Z}_{n}^* , b\in \mathbb{Z}_n \Biggr\}.
\end{equation*}
\end{theorem}
In the next Theorem, we exploit the proof of Theorem 2.3  which is based on the fact that each automorphism $\varphi$ of $D_n$ is determined by the image of rotation $r$ and reflection $s$. More precisely, \begin{equation*}
    \varphi(r)=r^a \hspace{0.1cm}\text{and}\hspace{0.1cm} \varphi(s)=r^bs,\hspace{0.1cm} \text{where} \hspace{0.1cm} a\in \mathbb{Z}_n^*, b \in \mathbb{Z}_n.
    \end{equation*}
Note that from Lemma 2.1, AutCl($D_n$) is a poset. The following result establish that $\text{AutCl}(G)$ is a lattice, if $G=D_n$.
\begin{theorem}
    The poset $\text{AutCl}(D_n)$ is a lattice for all positive integer $n$. 
    \end{theorem}
\begin{proof}
     The result holds trivially for $n=1,2$. For $n\ge 3$, let $n=p_1^{t_1}p_2^{t_2}\dots p_k^{t_k}$, where $p_i$'s are distinct primes and $t_i > 0$, for $1\le i \le k$. From the Remark 1, every subgroup of $D_n$ is either of type (1) or of type (2). Let $[H_1]$ and $[H_2]$ be two elements of AutCl$(D_n)$.  To prove AutCl($D_n$) is a lattice, it is sufficient to show that the meet and join of $[H_1]$ and $[H_2]$ exists.\par Consider the following cases: \vspace{0.2cm} \\
\textbf{Case 1:} If both $H_1$ and $H_2$ are of type (1), then $H_1=\bigl<r^{p_1^{u_1}p_2^{u_2}\dots p_k^{u_k}}\bigr>$ and $H_2=\bigl<r^{p_1^{v_1}p_2^{v_2}\dots p_k^{v_k}}\bigr>$, with $0\le u_i, v_i\le t_i$, $1\le i \le k$. We show that,
\begin{equation*}
    [H_1]\vee' [H_2]=[K_1] \hspace{0.2cm} \text{and} \hspace{0.2cm} [H_1]\wedge' [H_2]=[K_2],
\end{equation*}
where 
\begin{equation*}
 K_1=   \bigl<r^{p_1^{\text{min}\{u_1,v_1\}}p_2^{\text{min}\{u_2,v_2\}}\dots p_k^{\text{min}\{u_k,v_k\}}}\bigr> \hspace{0.2cm} \text{and }\hspace{0.2cm} K_2=\bigl<r^{p_1^{\text{max}\{u_1,v_1\}}p_2^{\text{max}\{u_2,v_2\}}\dots p_k^{\text{max}\{u_k,v_k\}}}\bigr>. 
\end{equation*}
Clearly, $[H_1],[H_2]\lesssim[K_1]$ as $H_1, H_2\le K_1$. So, $[K_1]$ is an upper bound of $\{[H_1],[H_2]\}$. In order to show that $[K_1]$ is the least upper bound of $\{[H_1],[H_2]\}$, assume that $[\bar H]$ be an upper bound of $\{[H_1],[H_2]\}$, then there exist $\phi_1,\phi_2 \in \text{Aut}(D_n)$ with $\phi_1(H_1), \phi_2(H_2)\subseteq \bar H$. Moreover, by Theorem 2.3, $[H_1],[H_2]$ are singletons as $H_1, H_2$ are of type (1). Therefore, $\phi_1(H_1)=H_1, \phi_2(H_2)=H_2$ and consequently, $K_1=H_1 \vee H_2 \subseteq \bar H$, which implies $[K_1]\lesssim [\bar H]$.  \par
Clearly, $[K_2]\lesssim [H_1], [H_2]$ as $K_2\le H_1, H_2$ and so, $[K_2]$ is a lower bound of $\{[H_1],[H_2]\}$. To show that $[K_2]$ is the greatest lower bound of $\{[H_1],[H_2]\}$, we assume that $[\widehat H]$ be a lower bound of  $\{[H_1],[H_2]\}$, then there are maps $\phi_1',\phi_2'\in \text{Aut}(D_n)$ with $\phi_1'(\widehat H)\subseteq H_1$ and $\phi_2' (\widehat H)\subseteq H_2$. By Remark 1, $\widehat H$ is of type (1), thus, by Theorem 2.3, $[\widehat H]$ is singleton and $\phi_1'(\widehat H),\phi_2'(\widehat H)=\widehat H$, which implies $\widehat H \le H_1\wedge H_2 =K_2$ and consequently, $[\widehat H]\lesssim [K_2]$. \vspace{0.2cm}\\
\textbf{Case 2:} If $H_1$ is of type (1) and $H_2$ is of type (2) generated by reflections only, then $H_1=$ $\bigl<r^{p_1^{u_1}p_2^{u_2}\dots p_k^{u_k}}\bigr>$ and $H_2=\bigl<r^js\bigr>$, $0\le u_i\le t_i$, $1\le i\le k$ and $0\le j \le n-1$. Note that $[\bigl<r^js\bigr>]=[\bigl<s\bigr>]$, so without the loss of generality, we can choose $H_2=\bigl<s\bigr>$. We show that,
\begin{equation*}
    [H_1]\vee' [H_2]=[K_1] \hspace{0.2cm} \text{and} \hspace{0.2cm} [H_1]\wedge' [H_2]=[K_2],
\end{equation*}
where  
\begin{equation*}
 K_1=   \bigl<r^{p_1^{u_1}p_2^{u_2}\dots p_k^{u_k}},s\bigr> \hspace{0.2cm} \text{and }\hspace{0.2cm} K_2=\{e\}. 
\end{equation*}
Clearly, $[H_1],[H_2]\lesssim [K_1]$ as $H_1,H_2\le K_1$. In order to show that $[K_1]$ is the least upper bound of $\{[H_1],[H_2]\}$, we assume that $[\bar H]$ be an upper bound of $\{[H_1],[H_2]\}$, then there exist $\phi_1,\phi_2 \in \text{Aut}(D_n)$ with $\phi_1(H_1), \phi_2(H_2)\subseteq \bar H$. By Theorem 2.3, $[H_1]$ is singleton, so, $\phi_1(H_1)=H_1\le \bar H$ and $\phi_2(H_2)=$ $ \bigl<r^js\bigr>\le \bar H$, for some $j$, so, $\phi_1(H_1)\vee \phi_2(H_2)\in [K_1]$ and consequently, $[K_1]\lesssim [\bar H]$. \par
Certainly, $[K_2]\lesssim [H_1], [H_2]$ as $K_2=\{e\}$. So $[K_2]$ is a lower bound of $\{[H_1],[H_2]\}$. Let $[\widehat H]$ be a lower bound of $\{[H_1],[H_2]\}$, then $\widehat H$ consists of rotations only, as $[\widehat H]\lesssim [H_1]$, also, $\widehat H$ consists of reflections only, as $[\widehat H]\lesssim [H_2]$. Therefore, $\widehat H=\{e\}$ and hence $[\widehat H]\lesssim [K_2]$. \vspace{0.2cm}\\
\textbf{Case 3}: If $H_1$ is of type (1) and $H_2$ is of type (2), then $H_1=\bigl<r^{p_1^{u_1}p_2^{u_2}\dots p_k^{u_k}}\bigr>$ and $H_2=$$ \bigl<r^{p_1^{v_1}p_2^{v_2}\dots p_k^{v_k}},r^js\bigr>$, $0\le u_i,v_i\le t_i$, $1\le i\le k$ and $0\le j\le p_1^{v_1}p_2^{v_2}\dots p_k^{v_k}-1$. Note that, $ [\bigl<r^{p_1^{v_1}p_2^{v_2}\dots p_k^{v_k}},r^js\bigr>]=[\bigl<r^{p_1^{v_1}p_2^{v_2}\dots p_k^{v_k}},s\bigr>]$, so, without the loss of generality, we can choose $H_2=\bigl<r^{p_1^{v_1}p_2^{v_2}\dots p_k^{v_k}},s\bigr>$. We show that,
\begin{equation*}
    [H_1]\vee' [H_2]=[K_1] \hspace{0.2cm} \text{and} \hspace{0.2cm} [H_1]\wedge' [H_2]=[K_2],
\end{equation*}
where
\begin{equation*}
 K_1=  \bigl <r^{p_1^{\text{min}\{u_1,v_1\}}p_2^{\text{min}\{u_2,v_2\}}\dots p_k^{\text{min}\{u_k,v_k\}}},s\bigr> \hspace{0.2cm} \text{and }\hspace{0.2cm} K_2=\bigl<r^{p_1^{\text{max}\{u_1,v_1\}}p_2^{\text{max}\{u_2,v_2\}}\dots p_k^{\text{max}\{u_k,v_k\}}}\bigr>.
\end{equation*}
Clearly, $[H_1],[H_2]\lesssim [K_1]$ as $H_1,H_2\le K_1$. Let $[\bar H]$ be an upper bound of $\{[H_1],[H_2]\}$, then there are maps $\phi_1,\phi_2\in \text{Aut}(G)$, with $\phi_1(H_1), \phi_2(H_2)\subseteq \bar H$. As $H_1$ is of type (1), by Theorem 2.3, the class $[H_1]$ is singleton, so, $\phi_1(H_1)=H_1\le \bar H$ and also $\phi_2(\bigl<r^{p_1^{v_1}\dots p_k^{v_k}}\bigr>)=\bigl<r^{p_1^{v_1}\dots p_k^{v_k}}\bigr>\le \bar H$. Clearly, $\phi_2(s)\in \bar H$ is a reflection, so, $\phi_1(H_1)\vee\phi_2(H_2)=H_1\vee \bigl<r^{p_1^{v_1}\dots p_k^{v_k}},\phi_2(s)\bigr>\le \bar H$ and hence, $[\phi_1(H_1)\vee \phi_2(H_2)]= [K_1]$, which implies $[K_1]\lesssim [\bar H]$ and consequently, $[K_1]$ is the least upper bound of $\{[H_1],[H_2]\}$. \par
Certainly, $[K_2]\lesssim [H_1], [H_2]$ as $K_2\le H_1,H_2$. So, $[K_2]$ is a lower bound of $\{[H_1],[H_2]\}$. Let $[\widehat H]$ be a lower bound of $\{[H_1],[H_2]\}$, then $\widehat H $ contains rotations only as $[\widehat H]\lesssim [H_1]$, so, $[\widehat H]$ is singleton and consequently, $\widehat H \le H_1, H_2$, which implies $\widehat H \le H_1\wedge H_2=K_2$. So, $[\widehat H]\lesssim [K_2]$ and hence, $[K_2]$ is the greatest lower bound of $\{[H_1],[H_2]\}$.
\vspace{0.2cm} \\ 
\textbf{Case 4:} If both $H_1$ and $H_2$ are of type (2), then without the loss of generality, assume that $H_1=\bigl<r^{p_1^{u_1}p_2^{u_2},\dots p_k^{u_k}},s\bigr>$ and $H_2=\bigl<r^{p_1^{v_1}p_2^{v_2}\dots p_k^{v_k}},s\bigr>$, $0\le u_i,v_i\le t_i$,$1\le i\le k$. Then, we show that,
\begin{equation*}
    [H_1]\vee' [H_2]=[K_1] \hspace{0.2cm} \text{and} \hspace{0.2cm} [H_1]\wedge' [H_2]=[K_2],
\end{equation*}
where 
\begin{equation*}
 K_1=   \bigl<r^{p_1^{\text{min}\{u_1,v_1\}}p_2^{\text{min}\{u_2,v_2\}}\dots p_k^{\text{min}\{u_k,v_k\}}},s\bigr> \hspace{0.2cm} \text{and }\hspace{0.2cm} K_2=\bigl<r^{p_1^{\text{max}\{u_1,v_1\}}p_2^{\text{max}\{u_2,v_2\}}\dots p_k^{\text{max}\{u_k,v_k\}}},s\bigr>. 
\end{equation*}
Clearly, $[H_1],[H_2]\lesssim [K_1]$ as $H_1,H_2\le K_1$. So, $[K_1]$ is an upper bound of $\{[H_1],[H_2]\}$. In order to show that $[K_1]$ is the least upper bound of $\{[H_1],[H_2]\}$, we assume that $[\bar H]$ be an upper bound of $\{[H_1],[H_2]\}$, then $\bigl<r^{p_1^{u_1}\dots p_k^{u_k}}\bigr>,\bigl<r^{p_1^{v_1}\dots p_k^{v_k}}\bigr>\le \bar H$ and also $\bar H$ contains reflections as automorphisms map reflection $s$ to a reflection, say, $r^js$, for some $j$. Therefore, $[K_1]\lesssim [\bar H]$ and hence, $[K_1]$ is the least upper bound of $\{[H_1],[H_2]\}$.\par
Clearly, $[K_2]\lesssim [H_1],[H_2]$ as $K_2\le H_1,H_2$. So, $[K_2]$ is a lower bound of $\{[H_1],[H_2]\}$. To show that $[K_2]$ is the greatest lower bound of $\{[H_1],[H_2]\}$, we assume that $[\widehat H]$ be any lower bound of  $\{[H_1],[H_2]\}$, then there is an automorphic image of $\widehat H$ in $H_1$ and $H_2$. This means that, there is an automorphic image of $\widehat H$ in $K_2$. Thus, by the definition of AutCl($D_n$), $[\widehat H]\lesssim [K_2]$ and consequently, $[K_2]$ is the greatest lower bound of $\{[H_1],[H_2]\}$. 
\end{proof}
In the following theorem, we have shown that $\text{AutCl}(D_n)$ is a lattice of known type, for some particular values of $n$.
\begin{theorem}
For a prime $p$, \\
\begin{center}
$\text{AutCl}(D_{p^\alpha})\cong$ $\begin{cases}
    C_3, & \text{if} \hspace{0.2cm} p=2 \hspace{0.2cm} \text{and} \hspace{0.2cm}\alpha =1, \\
    M_2, & \text{if} \hspace{0.2cm} p \ne 2 \hspace{0.2cm}\text{and} \hspace{0.2cm} \alpha =1,\\
    T(p_1^\alpha p_2), & \text{if} \hspace{0.2cm} \alpha \ge 2, \hspace{0.2cm} \text{where $p_1,p_2$ are any distinct primes.}
\end{cases}$
\end{center}
  Furthermore, $\text{AutCl}(D_{p^\alpha})$ contains $2(\alpha+1)$ elements, whenever $\alpha$ is a positive integer $\ge 2$. Moreover, for distinct primes $p_1$ and $p_2$, $\text{AutCl}(D_{p_1p_2})$ is isomorphic to the lattice of power set of 3 elements.
\end{theorem}
\begin{proof}
     For $p=2$ and $\alpha=1$, we have, $\bigl<r\bigr>,\bigl<s\bigr>$ and $\bigl<rs\bigr>$ belongs to the same class, as they are the images of $\bigl<r\bigr>$ under the automorphisms $\phi_1,\phi_2$ and $\phi_3$, respectively, where $\phi_1(r)=r$ and $\phi_1(s)=s$, $\phi_2(r)=s$ and $\phi_2(s)=r$, $\phi_3(r)=rs$ and $\phi_3(s)=s$. So, distinct elements of $\text{AutCl}(D_2)$ are $[\bigl<e\bigr>],[\bigl<r\bigr>],[D_2]$ with $[\bigl<e\bigr>]\lesssim[\bigl<r\bigr>]\lesssim[D_2]$ and hence $\text{AutCl}(D_2)\cong C_3$.\par
For odd prime $p$, the distinct elements of $\text{AutCl}(D_p)$ are $[\bigl<e\bigr>],[\bigl<r\bigr>],[\bigl<s\bigr>],[D_p]$, as the order of subgroups $\bigl<e\bigr>,\bigl<r\bigr>,\bigl<s\bigr>,D_p$ are all distinct. Furthermore, $[\bigl<e\bigr>]\lesssim [\bigl<r\bigr>]\lesssim [D_p]$ and $[\bigl<e\bigr>]\lesssim [\bigl<s\bigr>]\lesssim[D_p]$ and, $[\bigl<r\bigr>]$ and $[\bigl<s\bigr>]$ are incomparable, as under automorphism the subgroup generated by rotation maps to subgroup generated by rotation and same for reflections and consequently, $\text{AutCl}(D_p)\cong M_2$. \par
In $\text{AutCl}(D_{p^\alpha})$, there are $\alpha+1$ distinct classes containing subgroups of type (1), viz., $[\bigl<e\bigr>],$  $ [\bigl<r\bigr>],[\bigl<r^p\bigr>],[\bigl<r^{p^2}\bigr>],\dots,[\bigl<r^{p^{\alpha-1}}\bigr>]$ as order of each of class representatives are distinct. Furthermore, for the subgroup $\bigl<r^is\bigr>$, the map, $r\to r, s\to r^is$ is an automorphism that maps $\bigl<s\bigr>$ to $\bigl<r^is\bigr>$. So, all subgroups of $D_{p^\alpha}$ generated by reflections are contained in the class $[\bigl<s\bigr>]$. Also, for any class $[\bigl<r^{p^i}\bigr>]$, $1\le i\le \alpha-1$, there exists a class $[\bigl<r^{p^i},s\bigr>   ]$ containing subgroup of type (2). So, there are $\alpha-1$ distinct classes of the form $[\bigl<r^{p^i},s\bigr>]$ and lastly there is a class $[D_{p^\alpha}]$. Therefore, the total number of elements of AutCl($D_{p^\alpha}$) are $(\alpha+1)+1+(\alpha-1)+1=2(\alpha+1)$. \par For $\alpha\ge 2$, consider the map $\varphi: \text{AutCl}(D_{p^{\alpha}})\to T(p_1^\alpha p_2)$ given by $\varphi([\bigl<r^{p^{\alpha-j}}\bigr>])=p_1^j$ and $\varphi([\bigl<r^{p^{\alpha-j}},s\bigr>])=p_1^jp_2$. The map $\varphi$ is a lattice isomorphism between $\text{AutCl}(D_{p^{\alpha}})$ and $T(p_1^\alpha p_2)$. Thus, $\text{AutCl}(D_{p^{\alpha}})\cong T(p_1^\alpha p_2)$. \par
Now, in $\text{Aut}(D_{p_1p_2})$, it is clear that the elements $[\bigl<e\bigr>],[\bigl<r^{p_1}\bigr>],[\bigl<r^{p_2}\bigr>],$  $ [\bigl<r\bigr>],[\bigl<s\bigr>],$ $[\bigl<r^{p_1},s\bigr>],[\bigl<r^{p_2},s\bigr>],[D_{p_1p_2}]$ are all distinct as the order of their representatives are distinct. Let $\wp(X)$ be the power set of $X=\{1,2,3\}$. Then the map $\varphi:\text{AutCl}(D_{p_1p_2})\to \wp(X)$ given by $\varphi([\bigl<e\bigr>])=\{\}$ the empty set, $\varphi([\bigl<r^{p_1}\bigr>])=\{1\}$, $\varphi([\bigl<r^{p_2}\bigr>])=\{2\}$, $\varphi([\bigl<s\bigr>])=\{3\}$, $\varphi([\bigl<r^{p_1},s\bigr>])=\{1,3\}$, $\varphi([\bigl<r^{p_2},s\bigr>])=\{2,3\}$, $\varphi([\bigl<r\bigr>])=\{1,2\}$, $\varphi([D_{p_1p_2}])=\{1,2,3\}$ is a lattice isomorphism and consequently, $\text{AutCl}(D_{p_1p_2})\cong \wp(X)$. \end{proof}
In order to show that AutCl$(D_n)$ is a distributive lattice, we essentially use the following characterization due to Birkhoff \cite{Schmidt+1994}. 
\begin{theorem}
     \cite{Schmidt+1994} A lattice is distributive if and only if it does not contain a sublattice isomorphic to a pentagon $(N_5)$ or a diamond $(M_3)$. \end{theorem}
\begin{theorem}
 For positive integer $n$, the lattice $\text{AutCl}(D_n)$ does not contain a sublattice isomorphic to pentagon ($N_5$). \end{theorem}
\begin{proof}
     For $n=1$, we have, $D_1\cong \mathbb{Z}_2$, so, AutCl($D_1$) $ \cong$ AutCl($\mathbb{Z}_2$) $\cong$ $C_2$, also, if $n=2$, we have $D_2\cong \mathbb{Z}_2 \times \mathbb{Z}_2$, therefore, $\text{AutCl}(D_2)\cong C_3$, so, the result is true for $n=1,2$. \par Now, for $n\ge 3$, let $n=p_1^{t_1}p_2^{t_2}\dots p_k^{t_k}$ be the prime factorization of $n$. If there exists a sublattice of AutCl($D_n$) isomorphic to $N_5$, then there are distinct elements $[H_1],[H_2],[H_3],[H_1]\vee'[H_2],[H_1]\wedge' [H_2] \in \text{AutCl}(D_n)$ as depicted in Figure 2. 
\begin{figure}[h!]
    \centering
    \begin{tikzpicture}[scale=1.5,  every node/.style={circle, fill=black, minimum size=6pt, inner sep=0pt}]
				
				% Nodes in set U
				\node[label=left:${[H_2]}$] (u2) at (0,1) {};
				\node[label=left:${[H_1]}$] (u1) at (0,0.5) {};
				\node[label=right:${[H_1]\vee'[H_3]}$] (u5) at (0.5,1.5) {};
                \node[label=right:${[H_3]}$] (u3) at (1,0.75) {};
				\node[label=right:${[H_1]\wedge'[H_3]}$] (u4) at (0.5,0) {};
				
				% Edges
				\draw (u1) -- (u2);
				\draw (u2) -- (u5);
                \draw (u1) -- (u4);
                \draw (u4) -- (u3);
                \draw (u3) -- (u5);

		\end{tikzpicture}
        \caption{}
\end{figure}\\
Now, consider the following cases:\vspace{0.2cm} \\
\textbf{Case 1:} If $H_1$ and $H_3$ are subgroups of $D_n$ of type (1), where $H_1=\bigl<r^{p_1^{u_1}p_2^{u_2}\dots p_k^{u_k}}\bigr>$ and $H_3=$ $\bigl<r^{p_1^{v_1}p_2^{v_2}\dots p_k^{v_k}}\bigr>$, $0\le u_i,v_i\le t_i$, $1\le i\le k$, then $[H_1]\vee' [H_3]=$    $[H_2]\vee' [H_3]=[K]$, where $K=\bigl<r^{p_1^{\text{min}\{u_1,v_1\}}p_2^{\text{min}\{u_2,v_2\}}\dots p_k^{\text{min}\{u_k,v_k\}}}\bigr> $. Since $K$ is a subgroup of $D_n$ containing rotations only, by Theorem 2.3, the class $[K]$ is singleton. Therefore, $[H_2]$ is also a singleton containing $H_2$, which is of type (1). Thus, $H_2=\bigl<r^{p^{l_1}p_2^{l_2}\dots p_k^{l_k}}\bigr>$, where $\text{min}\{u_i,v_i\}\le l_i\le u_i$. As, $[H_2]\vee' [H_3]=[H_1]\vee' [H_3]$, we have $\bigl<r^{p_1^{\text{min}\{l_1,v_1\}}p_2^{\text{min}\{l_2,v_2\}}\dots p_k^{\text{min}\{l_k,v_k\}}}\bigr>=   \bigl<r^{p_1^{\text{min}\{u_1,v_1\}}p_2^{\text{min}\{u_2,v_2\}}\dots p_k^{\text{min}\{u_k,v_k\}}}\bigr> $. On comparing the order of generators, we get,
\begin{equation*}
\frac{n}{\gcd(n, p_1^{\text{min}\{l_1,v_1\}}\dots p_k^{\text{min}\{l_k,v_k\}}   )}=\frac{n}{\gcd(n, p_1^{\text{min}\{u_1,v_1\}}\dots p_k^{\text{min}\{u_k,v_k\}})}
\end{equation*}
and which implies $p_1^{\text{min}\{l_1,v_1\}}\dots p_k^{\text{min}\{l_k,v_k\}}=p_1^{\text{min}\{u_1,v_1\}}\dots p_k^{\text{min}\{u_k,v_k\}}$. Therefore, $\text{min}\{l_i,v_i\}=\text{min}\{u_i,v_i\}$, for all $i$.
Moreover, as $[H_2]\wedge' [H_3]=[H_1]\wedge'[H_3]$, we have $\bigl<r^{p_1^{\text{max}\{l_1,v_1\}}p_2^{\text{max}\{l_2,v_2\}}\dots p_k^{\text{max}\{l_k,v_k\}}}\bigr>=  $ $ \bigl<r^{p_1^{\text{max}\{u_1,v_1\}}p_2^{\text{max}\{u_2,v_2\}}\dots p_k^{\text{max}\{u_k,v_k\}}}\bigr>$. On comparing the order of generators, we get,
\begin{equation*}
\frac{n}{\gcd(n, p_1^{\text{max}\{l_1,v_1\}}\dots p_k^{\text{max}\{l_k,v_k\}}   )}=\frac{n}{\text{gcd}(n, p_1^{\text{max}\{u_1,v_1\}}\dots p_k^{\text{max}\{u_k,v_k\}})}
\end{equation*}
and which implies $p_1^{\text{max}\{l_1,v_1\}}\dots p_k^{\text{max}\{l_k,v_k\}}=p_1^{\text{max}\{u_1,v_1\}}\dots p_i^{\text{max}\{u_k,v_k\}}$. Therefore, $\text{max}\{l_i,v_i\}=\text{max}\{u_i,v_i\}$, for all $i$. This implies $l_i=u_i$ for all $i$, and hence $[H_1]=[H_2]$, a contradiction. \vspace{0.2cm} \\
\textbf{Case 2:} If $H_1$ is of type (1) and $H_3$ is of type (2) containing only reflection, then $H_1=$ $ \bigl<r^{p_1^{u_1}p_2^{u_2}\dots p_k^{u_k}}\bigr>$, $0\le u_i\le t_i$, $1\le i\le k$, and without the loss of generality, $H_3$ can be chosen to be $\bigl<s\bigr>$. So, by the Theorem 2.4, $[H_1]\vee' [H_3]=[H_2]\vee' [H_3]=$ $ [\bigl<r^{p_1^{u_1}p_2^{u_2}\dots p_k^{u_k}},s\bigr>]$ and $[H_1]\wedge' [H_2]=[H_2]\wedge'[H_3]=[\bigl<e\bigr>]$. Since $[H_1]\lesssim [H_2]$ and as $[H_1]$ is singleton, so $H_1\subseteq H_2$. Now, by Theorem 2.4, $H_2$ does not contain any reflection, thus, $[H_1]=[H_2]$, which is a contradiction. \vspace{0.2cm}\\
\textbf{Case 3:} If $H_1$ is of type (2) containing only reflection and $H_3$ is of type (1), then without the loss of generality, choose $H_1=\bigl<s\bigr>$ and $H_3=\bigl<r^{p_1^{u_1}p_2^{u_2}\dots p_k^{u_k}}\bigr>$, $0\le u_i\le t_i$, $1\le i\le k$. So, by the Theorem 2.4, $[H_1]\wedge' [H_3]=[H_2]\wedge' [H_3]=[\bigl<e\bigr>]$ and $[H_1]\vee'[H_3]=[H_2]\vee'[H_3]=$ $[\bigl<r^{p_1^{u_1}p_2^{u_2}\dots p_k^{u_k}},s\bigr>] $. As, $[H_1]\lesssim [H_2]$, so by Theorem 2.3, $H_2$ contains a reflection, say $r^is$, for some $i$, and thus the class $[H_2]$ is same as the class $[\bigl<r^{p_1^{l_1}p_2^{l_2}\dots p_k^{l_k}},s\bigr>]$, for some $l_i$, with $0\le l_i \le t_i$, $1\le i\le k$. As, $[H_2]\vee'[H_3]=[H_1]\vee'[H_3]$, we have $[\bigl<r^{p_1^{\text{min}\{l_1,u_1\}}p_2^{\text{min}\{l_2,u_2\}}\dots p_k^{\text{min}\{l_k,u_k\}}},s\bigr>]= $ $  [\bigl<r^{p_1^{u_1}p_2^{u_2}\dots p_k^{u_k}},s\bigr>] $ and hence by Theorem 2.3, $\bigl<r^{p_1^{\text{min}\{l_1,u_1\}}p_2^{\text{min}\{l_2,u_2\}}\dots p_k^{\text{min}\{l_k,u_k\}}}\bigr>=
\bigl<r^{p_1^{u_1}p_2^{u_2}\dots p_k^{u_k}}\bigr>$. On comparing the order of generators, we get,
\begin{equation*}
\frac{n}{\gcd(n, p_1^{\text{min}\{l_1,u_1\}}\dots p_k^{\text{min}\{l_k,u_k\}}   )}=\frac{n}{\gcd(n, p_1^{u_1}\dots p_k^{u_k})}
\end{equation*}
and which implies $p_1^{\text{min}\{l_1,u_1\}}\dots p_k^{\text{min}\{l_k,u_k\}}=p_1^{u_1}\dots p_k^{u_k}$. Therefore, $\text{min}\{l_i,u_i\}=u_i$, for all $i$, so, $u_i\le l_i$, for all $i$. Similarly, as $[H_2]\wedge'[H_3]=[H_1]\wedge'[H_3]$, we have $\bigl<r^{p_1^{\text{max}\{l_1,u_1\}}p_2^{\text{max}\{l_2,u_2\}}\dots p_k^{\text{max}\{l_k,u_k\}}}\bigr> $ $=\bigl<r^{p_1^{l_1}p_2^{l_2}\dots p_k^{l_k}}\bigr>=\bigl<e\bigr>$, and therefore, $[H_2]=\bigl<s\bigr>=[H_1]$, a contradiction.
\vspace{0.2cm} \\
\textbf{Case 4:} If $H_1$ is of type (1) and $H_3$ is of type (2), then $H_1=\bigl<r^{p_1^{u_1}p_2^{u_2}\dots p_k^{u_k}}\bigr>$ and without the loss of generality, assume that $H_3=$ $\bigl<r^{p_1^{v_1}p_2^{v_2}\dots p_k^{v_k}},s\bigr>$, $0\le u_i,v_i\le t_i$, $1\le i\le k$. Clearly, $[H_1]\vee' [H_3]=[H_2]\vee' [H_3]= [\bigl<r^{p_1^{\text{min}\{u_1,v_1\}}p_2^{\text{min}\{u_2,v_2\}}\dots p_k^{\text{min}\{u_k,v_k\}}},s\bigr>]$ and $[H_1]\wedge' [H_3]=[H_2]\wedge' [H_3]=$ 
$[\bigl<r^{p_1^{\text{max}\{u_1,v_1\}}p_2^{\text{max}\{u_2,v_2\}}\dots p_k^{\text{max}\{u_k,v_k\}}}\bigr>]$. This implies no subgroup in $[H_2]$ contains reflections. So, consider $H_2=\bigl<r^{p_1^{l_1}p_2^{l_2}\dots p_k^{l_k}}\bigr>$ for some $l_i$, $1\le i\le k$ and let $K_l=H_l \backslash \{r^is\hspace{0.2cm}|\hspace{0.2cm} 0\le i \le n-1\} $, for $l=1,2,3$, then clearly, $K_l\le H_l$ and by the Theorem 2.4, $[K_1]\vee' [K_3]=[K_2]\vee'[K_3]=[\bigl<r^{p_1^{\min\{u_1,v_1\}}p_2^{\min\{u_2,v_2\}}\dots p_k^{\min\{u_k,v_k\}}}\bigr>]$ and $[K_1]\wedge' [K_3]=[K_2]\wedge'[K_3]=[\bigl<r^{p_1^{\text{max}\{u_1,v_1\}}p_2^{\text{max}\{u_2,v_2\}}\dots p_k^{\text{max}\{u_k,v_k\}}}\bigr>]$. As, $K_1=H_1$ and $K_2=H_2$, so $[K_1]$ and $[K_2]$ are distinct and $[K_1]\lesssim[K_2]$. Certainly, $[K_3]\lesssim[K_2]$ is not possible, as if $[K_3]\lesssim [K_2]$, then this would imply, $[K_1]\wedge'[K_3]=[K_2]\wedge'[K_3]=[K_3]$, which implies $[K_3]\lesssim [K_1]$, so, $[H_1]=[K_1]=[K_1]\vee'[K_3]=[K_2]\vee'[K_3]=[K_2]=[H_2]$, a contradiction. Similarly, $[K_3]\lesssim [K_1]$ is not possible. Furthermore, $[K_2]\lesssim [K_3]$ is not possible, if $[K_2]\lesssim [K_3]$, then $[H_2]=[K_2]\lesssim[K_3]\lesssim [H_3]$, a contradiction, and similarly, $[K_1]\lesssim [K_3]$ is not possible. So, $[K_1],[K_2]$ and $[K_3]$ are distinct classes with $[K_1],[K_3]$ are incomparable and similarly $[K_2],[K_3]$ are incomparable. Therefore, $[K_1]\wedge'[K_3],[K_1]\vee'[K_3]$ are distinct from $[K_1]$ and $[K_3]$. Furthermore, $[K_1]\vee'[K_3]$ and $[K_1]\wedge'[K_3]$ are distinct, else if 
$[K_1]\vee'[K_3]=[K_1]\wedge'[K_3]$, then as, $[K_1]\wedge'[K_3]\lesssim[H_1]\lesssim[H_2]\lesssim[K_1]\vee'[K_3]$, which implies $[H_1]=[H_2]$, a contradiction. Certainly, $[K_1]\wedge'[K_3]$ is distinct from $[K_2]$ as $K_2=H_2$ and $[K_1]\wedge'[K_3]=[H_1]\wedge'[H_3]$. Also, $[K_1]\vee'[K_3]$ is distinct from $[K_2]$ else, $[K_2]=[K_1]\vee'[K_3]=[K_2]\vee'[K_3]$, which implies $[K_3]\lesssim [K_2]$, a contradiction. Thus, we have distinct $[K_1],[K_2],[K_3],[K_1]\vee'[K_3],[K_1]\wedge' [K_3]\in \text{AutCl}(D_n)$ as shown in Figure 3, which is not possible by case 1. 
\begin{figure}[h!]
    \centering
    \begin{tikzpicture}[scale=1.5,  every node/.style={circle, fill=black, minimum size=6pt, inner sep=0pt}]
				
				% Nodes in set U
				\node[label=left:${[K_2]}$] (u2) at (0,1) {};
				\node[label=left:${[K_1]}$] (u1) at (0,0.5) {};
				\node[label=right:${[K_1]\vee'[K_3]}$] (u5) at (0.5,1.5) {};
                \node[label=right:${[K_3]}$] (u3) at (1,0.75) {};
				\node[label=right:${[K_1]\wedge'[K_3]}$] (u4) at (0.5,0) {};
				
				% Edges
				\draw (u1) -- (u2);
				\draw (u2) -- (u5);
                \draw (u1) -- (u4);
                \draw (u4) -- (u3);
                \draw (u3) -- (u5);

		\end{tikzpicture}
        \caption{}
\end{figure}\\
\textbf{Case 5:} If $H_1$ is of type (2) and $H_3$ is of type (1), then without the loss of generality, assume that $H_1=\bigl<r^{p_1^{u_1}p_2^{u_2}\dots p_k^{u_k}},s\bigr>$ and $H_3=$$\bigl<r^{p_1^{v_1}p_2^{v_2}\dots p_k^{v_k}}\bigr>$, $0\le u_i,v_i\le t_i$, $1\le i\le k$. So, $[H_1]\vee' [H_3]=[H_2]\vee' [H_3]= [\bigl<r^{p_1^{\text{min}\{u_1,v_1\}}p_2^{\text{min}\{u_2,v_2\}}\dots p_k^{\text{min}\{u_k,v_k\}}},s\bigr>]$ and $[H_1]\wedge' [H_3]=[H_2]\wedge' [H_3]$ $=[\bigl<r^{p_1^{\text{max}\{u_1,v_1\}}p_2^{\text{max}\{u_2,v_2\}}\dots p_k^{\text{max}\{u_k,v_k\}}}\bigr>]$ and so all subgroups of $[H_2]$ contain reflections. Let $K_l=H_l \backslash \{r^is\hspace{0.2cm}|\hspace{0.2cm}0\le i\le n-1\} $, for $l=1,2,3$, then clearly $K_l\le H_l$ and by Theorem 2.4, $[K_1]\vee' [K_3]=[K_2]\vee'[K_3]=[\bigl<r^{p_1^{\min\{u_1,v_1\}}p_2^{\min\{u_2,v_2\}}\dots p_k^{\min\{u_k,v_k\}}}\bigr>]$ and $[K_1]\wedge' [K_3]=[K_2]\wedge'[K_3]=[\bigl<r^{p_1^{\text{max}\{u_1,v_1\}}p_2^{\text{max}\{u_2,v_2\}}\dots p_k^{\text{max}\{u_k,v_k\}}}\bigr>]$. Certainly, $[K_1]$ and $[K_2]$ are distinct as, $[H_1]$ and $[H_2]$ are distinct and $[K_1]\lesssim [K_2]$ as, $[H_1]\lesssim [H_2]$. Furthermore, $[K_3]\lesssim[K_2]$ is not possible, as if $[K_3]\lesssim[K_2]$, then $[H_3]=[K_3]\lesssim [K_2]\lesssim [H_2]$, a contradiction.
Also, $[K_2]\lesssim [K_3]$ is not possible, as if $[K_2]\lesssim [K_3]$, then $ [K_1]\vee'[K_3]=[K_2]\vee'[K_3]=[K_3]$, which implies $[K_1]\lesssim [K_3]$ and hence $[K_1]=[K_1]\wedge' [K_3]= [K_2]\wedge ' [K_3]= [K_2]$ and therefore, $[H_1]=[H_2]$, a contradiction and hence $[K_2]$ and $[K_3]$ are incomparable. Also, $[K_3]\lesssim [K_1]$ is not possible as, if $[K_3]\lesssim [K_1]$, then $[H_3]=[K_3]\lesssim [K_1]\lesssim[H_1]$, a contradiction. Clearly, $[K_1]\lesssim [K_3]$ is not possible, else, we have $[K_1]=[K_1]\wedge' [K_3]= [K_2]\wedge ' [K_3]= [K_2]$ and therefore, $[H_1]=[H_2]$, a contradiction and hence $[K_1]$ and $[K_3]$ are incomparable. Therefore, $[K_1]\wedge'[K_3],[K_1]\vee'[K_3]$ are distinct from $[K_1]$ and $[K_3]$. Also, $[K_1]\wedge'[K_3]$ and $[K_2]$ are distinct else, $[K_1]\wedge'[K_3]=[K_1]=[K_2]$, which implies $[H_1]=[H_2]$, a contradiction. Certainly, $[K_1]\vee'[K_3]$ is distinct from $[K_2]$ else, $[K_1]\vee'[K_3]=[K_2]\vee'[K_3]=[K_2]$, which implies $[K_3]\lesssim [K_2]$, a contradiction. Lastly, $[K_1]\wedge'[K_3]$ and $[K_1]\vee'[K_3]$ are distinct else, $[H_1]\wedge'[H_2]=[K_1]\wedge'[K_3]=[K_2]=[H_3]$, a contradiction. Thus, we have distinct $[K_1],[K_2],[K_3],[K_1]\vee'[K_3],[K_1]\wedge' [K_3]\in \text{AutCl}(D_n)$ as in Figure 3, which is not possible by case 1. \vspace{0.2cm}\\
\textbf{Case 6:} If both $H_1$ and $H_3$ are of type (2), then without the loss of generality, assume that $H_1=\bigl<r^{p_1^{u_1}p_2^{u_2}\dots p_k^{u_k}},s\bigr>$ and $H_3=$ $ \bigl<r^{p_1^{v_1}p_2^{v_2}\dots p_k^{v_k}},s\bigr>$, $0\le u_1,v_i\le t_i$, $1\le i\le k$. So, $[H_1]\vee' [H_3]=[H_2]\vee' [H_3]= [\bigl<r^{p_1^{\text{min}\{u_1,v_1\}}p_2^{\min\{u_2,v_2\}}\dots p_k^{\min\{u_k,v_k\}}},s\bigr>]$ and $[H_1]\wedge' [H_3]=[H_2]\wedge' [H_3]$ $=[\bigl<r^{p_1^{\text{max}\{u_1,v_1\}}p_2^{\text{max}\{u_2,v_2\}}\dots p_k^{\text{max}\{u_k,v_k\}}},s\bigr>]$. Let $K_l=H_l \backslash \{r^is\hspace{0.2cm}|\hspace{0.2cm} 0\le i\le n-1\} $, for $l=1,2,3$, then $K_l\le H_l$ and by Theorem 2.4, $[K_1]\vee' [K_3]= [K_2]\vee'[K_3]=[\bigl<r^{p_1^{\min\{u_1,v_1\}}p_2^{\min\{u_2,v_2\}}\dots p_k^{\min\{u_k,v_k\}}}\bigr>]$ and $[K_1]\wedge' [K_3]=[K_2]\wedge'[K_3]=[\bigl<r^{p_1^{\text{max}\{u_1,v_1\}}p_2^{\text{max}\{u_2,v_2\}}\dots p_k^{\text{max}\{u_k,v_k\}}}\bigr>]$. Furthermore, by the choices of $H_1,H_3$ and by Theorem 2.4, we have $[K_1],[K_2],[K_3],[K_1]\vee'[K_3],[K_1]\wedge' [K_3]$ are all distinct as in Figure 3, again which is not possible by case 1.\end{proof}
\begin{theorem}
     For a positive integer $n$, the lattice $\text{AutCl}(D_n)$ does not contain a sublattice isomorphic to a diamond ($M_3$). \end{theorem}
\begin{proof}
     The result holds trivially for $n=1,2$. For $n\ge 3$, let $n=p_1^{t_1}p_2^{t_2}\dots p_k^{t_k}$ be the prime factorization of $n$. Suppose that there exists a sublattice of AutCl($D_n$) isomorphic to $M_3$, then there are distinct elements $[H_1],[H_2],[H_3],[H_1]\vee'[H_2],[H_1]\wedge' [H_2] \in \text{AutCl}(D_n)$ as depicted in Figure 4.
\begin{figure}[h]
    \centering
    \begin{tikzpicture}[scale=1.5,  every node/.style={circle, fill=black, minimum size=6pt, inner sep=0pt}]
				
				% Nodes in set U
				\node[label=left:${[H_1]}$] (u1) at (-0.7,0) {};
				\node[label=right:${[H_3]}$] (u3) at (0.7,0) {};
				\node[label=left:${[H_1]\vee'[H_2]}$] (u4) at (0,0.7) {};
				\node[label=left:${[H_1]\wedge'[H_2]}$] (u5) at (0,-0.7) {};
                \node[label=left:${[H_2]}$] (u2) at (0,0) {};
            
				% Edges
				\draw (u1) -- (u4);
				\draw (u4) -- (u3);
				\draw (u3) -- (u5);
				\draw (u5) -- (u1);
                \draw (u5) -- (u2);
                \draw (u2) -- (u4);

		\end{tikzpicture}
        \caption{}
\end{figure}\\

Consider the following cases: \vspace{0.2cm} \\
\textbf{Case 1:} If both $H_1$ and $H_2$ are of type (1), then $H_1=\bigl<r^{p_1^{u_1}p_2^{u_2}\dots p_k^{u_k}}\bigr>$ and $H_2=$ $\bigl <r^{p_1^{v_1}p_2^{v_2}\dots p_k^{v_k}}\bigr>$, $0\le u_i,v_i\le t_i$, $1\le i\le k$, so, $[H_1]\vee' [H_2]=[H_1]\vee' [H_3]=[H_2]\vee'[H_3]=[K]$, where $K=\bigl<r^{p_1^{\min\{u_1,v_1\}}p_2^{\min\{u_2,v_2\}}\dots p_k^{\min\{u_k,v_k\}}}\bigr> $. Therefore, no subgroup in $[H_3]$ contains a reflection and so $[H_3]=[\bigl<r^{p_1^{l_1}p_2^{l_2}\dots p_k^{l_k}}\bigr>]$, for some $l_i$. As each classes are singleton and $[H_1]\vee' [H_2]=[H_1]\vee' [H_3]$, we have $\bigl<r^{p_1^{\min\{u_1,v_1\}}p_2^{\min\{u_2,v_2\}}\dots p_k^{\min\{u_k,v_k\}}}\bigr>$ = $\bigl<r^{p_1^{\min\{u_1,l_1\}}p_2^{\min\{u_2,l_2\}}\dots p_k^{\min\{u_k,l_k\}}}\bigr>$. On comparing the order of generators, we get, 
\begin{equation*}
\frac{n}{\gcd(n, p_1^{\min\{u_1,v_1\}}\dots p_k^{\min\{u_k,v_k\}})}=    \frac{n}{\gcd(n, p_1^{\min\{u_1,l_1\}}\dots p_k^{\min\{u_k,l_k\}}   )}
\end{equation*}
and which implies $p_1^{\min\{u_1,v_1\}}\dots p_i^{\min\{u_k,v_k\}}=p_1^{\min\{u_1,l_1\}}\dots p_k^{\min\{u_k,l_k\}}$. Therefore, $\min\{u_i,v_i\}=\min\{u_i,l_i\}$, for all $i$. Furthermore, as $[H_1]\wedge' [H_2]=[H_1]\wedge'[H_3]$, we have $\bigl<r^{p_1^{\text{max}\{u_1,v_1\}}p_2^{\text{max}\{u_2,v_2\}}\dots p_k^{\text{max}\{u_k,v_k\}}}\bigr>=  $ $ \bigl<r^{p_1^{\text{max}\{u_1,l_1\}}p_2^{\text{max}\{u_2,l_2\}}\dots p_k^{\text{max}\{u_k,l_k\}}}\bigr>$. On comparing the order of generators, we get,
\begin{equation*}
\frac{n}{\gcd(n, p_1^{\text{max}\{u_1,v_1\}}\dots p_k^{\text{max}\{u_k,v_k\}}   )}=\frac{n}{\text{gcd}(n, p_1^{\text{max}\{u_1,l_1\}}\dots p_k^{\text{max}\{u_k,l_k\}})}
\end{equation*}
and which implies $p_1^{\text{max}\{u_1,v_1\}}\dots p_k^{\text{max}\{u_k,v_k\}}=p_1^{\text{max}\{u_1,l_1\}}\dots p_i^{\text{max}\{u_k,l_k\}}$. Therefore, $\text{max}\{u_i,v_i\}=\text{max}\{u_i,l_i\}$, for all $i$, and this implies $l_i=v_i$, for all $i$, and consequently, $[H_2]=[H_3]$, a contradiction.\vspace{0.2cm}\\
\textbf{Case 2:} If $H_1$ is of type (1) and $H_2$ is of type (2) containing only rotation, then $H_1=$ $\bigl<r^{p_1^{u_1}p_2^{u_2}\dots p_k^{u_k}}\bigr>$ and without the loss of generality, assume that $H_2=\bigl<s\bigr>$. So, $[H_1]\vee'[H_2]=[\bigl<r^{p_1^{u_1}p_2^{u_2}\dots p_k^{u_k}},s\bigr> ]$ and $[H_1]\wedge' [H_2]=[\bigl<e\bigr>]$. As, $[H_2]\wedge' [H_3]=[H_1]\wedge' [H_2]=[\bigl<e\bigr>]$, so no subgroup of $[H_3]$ contains reflections and as $[H_2]\vee' [H_3]=[\bigl<r^{p_1^{u_1}p_2^{u_2}\dots p_k^{u_k}},s\bigr> ] $, we have $[H_3]=[H_1]$, which is a contradiction. \vspace{0.2cm}\\
\textbf{Case 3:}  If $H_1$ is of type (1) and $H_2$ is of type (2), then $H_1=\bigl<r^{p_1^{u_1}p_2^{u_2}\dots p_k^{u_k}}\bigr>$ and without the loss of generality, assume that $H_2=$ $\bigl<r^{p_1^{v_1}p_2^{v_2}\dots p_k^{v_k}},s\bigr>$, $0\le u_i,v_i\le t_i$, $1\le i\le k$. So, $[H_1]\vee' [H_2]=[H_1]\vee' [H_3]=[H_2]\vee'[H_3]=$ $[\bigl<r^{p_1^{\min\{u_1,v_1\}}p_2^{\min\{u_2,v_2\}}\dots p_k^{\min\{u_k,v_k\}}},s\bigr> ]$ and $[H_1]\wedge' [H_2]=[H_1]\wedge' [H_3]=[H_2]\wedge'[H_3]=$ $[\bigl<r^{p_1^{\text{max}\{u_1,v_1\}}p_2^{\text{max}\{u_2,v_2\}}\dots p_k^{\text{max}\{u_k,v_k\}}}\bigr> ]$, therefore, no subgroups of $[H_2]\wedge'[H_3]$ contains a reflection. Clearly, all subgroup of the class $[H_3]$ contains a reflection because subgroups in $[H_1]\vee'[H_3]$ contains reflections, this is because $[H_1]\vee' [H_2]=[H_1]\vee'[H_3]$ and $H_1$ is of type (1), which implies subgroups of $[H_2]\wedge'[H_3]$ also contains reflections, a contradiction. \vspace{0.2cm} \\
\textbf{Case 4:}  If $H_1$ is of type (2) and $H_2$ is of type (2) containing rotation only, then without the loss of generality, assume that $H_1=$ $\bigl<r^{p_1^{u_1}p_2^{u_2}\dots p_k^{u_k}},s\bigr>$ and $H_2=\bigl<s\bigr>$, $0\le u_i\le t_i$, $1\le i\le k$. So, $[H_1]\vee' [H_2]=[H_1]\vee' [H_3]=[H_2]\vee'[H_3]=[\bigl<r^{p_1^{u_1}p_2^{u_2}\dots p_k^{u_k}},s\bigr> ]=[H_1]$, a contradiction. \vspace{0.2cm}\\
\textbf{Case 5:} If both $H_1$ and $H_2$ are of type (2), then without the loss of generality, assume that $H_1=\bigl<r^{p_1^{u_1}p_2^{u_2}\dots p_k^{u_k}},s\bigr>$ and $H_2=$ $\bigl<r^{p_1^{v_1}p_2^{v_2}\dots p_k^{v_k}},s\bigr>$, $0\le u_i,v_i\le t_i$, $1\le i\le k$. So, $[H_1]\vee' [H_2]=[H_1]\vee' [H_3]=[H_2]\vee'[H_3]=[\bigl<r^{p_1^{\min\{u_1,v_1\}}p_2^{\min\{u_2,v_2\}}\dots p_k^{\min\{u_k,v_k\}}},s\bigr> ]$ and $[H_1]\wedge' [H_2]=[H_1]\wedge' [H_3]=[H_2]\wedge'[H_3]=$ $[\bigl<r^{p_1^{\text{max}\{u_1,v_1\}}p_2^{\text{max}\{u_2,v_2\}}\dots p_k^{\text{max}\{u_k,v_k\}}},s\bigr> ]$. Since all rotations of $D_n$ are closed under its operation, this case reduces to case 1. \end{proof}
\begin{remark}
    Note that in the proof of Theorem 2.7 and 2.8, whenever we chose a type (2) subgroup, without the loss of generality, we represented it by $\bigl<r^d,s\bigr>$ with $d|n$, instead of $\bigl<r^d,r^is\bigr>$ with $d|n$, $0\le i\le d-1 $, as $ [\bigl<r^d,s\bigr>]=$ $[\bigl<r^d,r^is\bigr>]$.
\end{remark}
\begin{corollary}
     $\text{AutCl}(D_n)$ is a modular lattice for all positive integer $n$. 
\end{corollary}
As groups of quaternions and generalized quaternions are particular classes of more generalized quaternion group $Q_{4m}$, it is interesting to work with $Q_{4m}$. The following Theorem describes the automorphism group of more generalized quaternions $Q_{4m}$.
\begin{theorem} \cite{conradquad}
    For $m\ge 3 $, \begin{equation*}
        \text{Aut}(Q_{4m})\cong \biggl\{\begin{pmatrix}
            a & b \\ 0 & 1
        \end{pmatrix}\hspace{0.2cm}|\hspace{0.2cm} a \in \mathbb{Z}_{2m}^*,\hspace{0.1cm} b\in \mathbb{Z}_{2m} \biggr\}    \end{equation*}
\end{theorem}
The proof of Theorem 2.9 is based on on the fact that each automorphism $\varphi$ of $Q_{4m}$ is determined by image of generators $x$ and $y$. More precisely, \begin{equation*}
    \varphi(x)=x^a \hspace{0.1cm}\text{and}\hspace{0.1cm} \varphi(y)=x^by,\hspace{0.1cm} \text{where} \hspace{0.1cm} a\in \mathbb{Z}_{2m}^*, b \in \mathbb{Z}_{2m}.
    \end{equation*}
    The following result establish that the poset $\text{AutCl}(G)$ is a lattice in the case of $G=Q_{4m}$.
    \begin{theorem}
        The poset $\text{AutCl}(Q_{4m})$ is a lattice for all positive integer $m$.
    \end{theorem}
    \begin{proof}
    By example 1 and 2, we have for $m=1,2$, $\text{AutCl}(Q_{4m})$ is a lattice. We will prove the result for the case when $2$ does not divide $m$ and the proof is similar when $2$ divides $m$. Let $m=p_1^{t_1}\dots p_{k}^{t_k}$ be the prime factorization of $m$. Consider the following cases: \vspace{0.2cm}\\
        \textbf{Case 1:} If $H_1=\bigl<x^{2^{\beta_1}{p_1^{u_1}}\dots p_k^{u_k}}\bigr>$ and $H_2=\bigl<x^{2^{\beta_2}{p_1^{v_1}}\dots p_k^{v_k}}\bigr>$ with $\beta_1,\beta_2 \in \{0,1\}$, then \begin{equation*}
            [H_1]\vee' [H_2]=[K_1] \hspace{0.2cm} \text{and }\hspace{0.2cm} [H_1]\wedge' [H_2]=[K_2], 
        \end{equation*}
        where \begin{equation*}
            K_1=\bigl<x^{2^{\min\{\beta_1,\beta_2\}}p_1^{\min\{u_1,v_1\}}\dots p_k^{\min\{u_k,v_k\}}}\bigr> \hspace{0.2cm} \text{and} \hspace{0.2cm} K_2=\bigl<x^{2^{\max\{\beta_1,\beta_2\}}p_1^{\max\{u_1,v_1\}}\dots p_k^{\max\{u_k,v_k\}}}\bigr>.
        \end{equation*}
        Clearly, $[K_1]$ is an upper bound of $\{[H_1],[H_2]\}$ as, $H_1,H_2\le K_1$. Let $[\bar H]$ be an upper bound of $\{[H_1],[H_2]\}$, then as $[H_1]$ and $[H_2]$ are singletons, we have $H_1,H_2\le \bar H$ and consequently, $[K_1]\lesssim [\bar H]$, which implies $[K_1] $ is the least upper bound of $\{[H_1],[H_2]\}$. \par
        Similarly, $[K_2]$ is a lower bound of $\{[H_1],[H_2]\}$ as $K_2\le H_1,H_2$. Let $[\widehat H]$ be a lower bound of $\{[H_1],[H_2]\}$, as $[H_1]$ and $[H_2]$ are singletons, so, $[\widehat H]$ is also singleton as $\widehat H\le \bigl<x\bigr>$ and consequently, $\widehat H \le H_1\wedge H_2=K_2$ which implies $[\widehat H]\lesssim [K_2]$ and hence, $[K_2]$ is the greatest lower bound of $\{[H_1],[H_2]\}.$ \vspace{0.1cm}\\
        \textbf{Case 2:} Let $H_1=\bigl<x^{2^{\beta}{p_1^{u_1}}\dots p_k^{u_k}}\bigr>$ and $H_2=\bigl<x^{p_1^{v_1}\dots p_k^{v_k}},y\bigr>$ with $\beta \in \{0,1\}$, $0\le u_i,v_i\le t_i$ and $1\le i\le k$. \vspace{0.2cm}\\
        \textbf{Subcase 2.1:} If $\beta=0$, then
        \begin{equation*}
            [H_1]\vee' [H_2]=[K_1] \hspace{0.2cm} \text{and }\hspace{0.2cm} [H_1]\wedge' [H_2]=[K_2], 
        \end{equation*}
        \begin{equation*}
            K_1=\bigl<x^{p_1^{\min\{u_1,v_1\}}\dots p_k^{\min\{u_k,v_k\}}},y\bigr> \hspace{0.2cm} \text{and} \hspace{0.2cm} K_2=\bigl<x^{p_1^{\max\{u_1,v_1\}}\dots p_k^{\max\{u_k,v_k\}}}\bigr>.
        \end{equation*}
       It is clear that $[K_1]$ is an upper bound of $\{[H_1],[H_2]\}$ as $H_1,H_2\le K_1$. Let $[\bar H]$ be an upper bound of $\{[H_1],[H_2]\}$, then $H_1\le \bar H$ and $\bigl<x^{{p_1^{v_1}}\dots p_k^{v_k}}\bigr>\le \bar H$. Moreover, $\bar H$ contains $x^iy$, for some $i$, as $y\in H_2$ and consequently $[K_1]\lesssim [\bar H]$, so, $[K_1]$ is the least upper bound of $\{[H_1],[H_2]\}$. \par
       Certainly, $[K_2]$ is a lower bound of $\{[H_1],[H_2]\}$. Let $[\widehat H]$ be a lower bound of $\{[H_1],[H_2]\}$, then $\widehat H\le H_1$ and as $H_1=\bigl<x^{2^\beta p_1^{u_1}\dots p_k^{u_k}}\bigr>$, so, by Theorem 2.9, $[\widehat H]$ is singleton, this implies $\widehat H \le H_2$ as, $[\widehat H]\lesssim [H_2]$, therefore, $\widehat H \le K_2$ and hence, $[\widehat H]\lesssim [K_2]$, so $[K_2]$ is the greatest lower bound of $\{[H_1],[H_2]\}$. \vspace{0.2cm}\\ 
        \textbf{Subcase 2.2:} If $\beta=1$, then
        \begin{equation*}
            K_1=\bigl<x^{p_1^{\min\{u_1,v_1\}}\dots p_k^{\min\{u_k,v_k\}}},y\bigr> \hspace{0.2cm} \text{and} \hspace{0.2cm} K_2=\bigl<x^{2{p_1^{\max\{u_1,v_1\}}\dots p_k^{\max\{u_k,v_k\}}}}\bigr>.
        \end{equation*} 
        On similar line, as in subcase 2.1, $[K_1]$ is the least upper bound of $\{[H_1],[H_2]\}$ and $[K_2]$ is the greatest lower bound of $\{[H_1],[H_2]\}$. \vspace{0.2cm}\\
        \textbf{Case 3:} If $H_1=\bigl<x^{p_1^{u_1}\dots p_k^{u_k}},y\bigr>$ and $H_2=\bigl<x^{p_1^{v_1}\dots p_k^{v_k}},y\bigr>$ with $\beta_1,\beta_2 \in \{0,1\}$, $0\le u_i,v_i\le t_i$ and $1\le i\le k$, then \begin{equation*}
            [H_1]\vee' [H_2]=[K_1] \hspace{0.2cm} \text{and }\hspace{0.2cm} [H_1]\wedge' [H_2]=[K_2], 
        \end{equation*}
        where \begin{equation*}
            K_1=\bigl<x^{p_1^{\min\{u_1,v_1\}}\dots p_k^{\min\{u_k,v_k\}}},y\bigr> \hspace{0.2cm} \text{and} \hspace{0.2cm} K_2=\bigl<x^{p_1^{\max\{u_1,v_1\}}\dots p_k^{\max\{u_k,v_k\}}},y\bigr>.
        \end{equation*}
        Clearly, $[K_1]$ is an upper bound of $\{[H_1],[H_2]\}$ as $H_1,H_2\le K_1$. Let $[\bar H]$ be an upper bound of $\{[H_1],[H_2]\}$, then $\bigl<x^{p_1^{u_1}}\dots p_k^{u_k}\bigr>,\bigl<x^{p_1^{v_1}}\dots p_k^{v_k}\bigr>\le \bar H$ and $x^jy\in \bar H$ as $y\in H_1$. So, $[K_1]\lesssim[\bar H]$ and hence, $[K_1]$ is the least upper bound of $\{[H_1],[H_2]\}$. \par
        Also, $[K_2]$ is a lower bound of $\{[H_1],[H_2]\}$ as $K_2\le H_1,H_2$. Let $[\widehat H]$ be a lower bound of $\{[H_1],[H_2]\}$ then $\widehat H \backslash \{x^iy\hspace{0.2cm }|\hspace{0.2cm}0\le i\le 2m-1\}\le \widehat H$ and by Theorem 2.9, $\widehat H \backslash \{x^iy\hspace{0.2cm }|\hspace{0.2cm}0\le i\le 2m-1\}\le K_2$. Clearly, $\bigl<y\bigr>\le K_2$ and as $[(\widehat H \backslash \{x^iy\hspace{0.2cm }|\hspace{0.2cm}0\le i\le 2m-1\})\vee \bigl<y\bigr>]=[\widehat H]$, we have $[\widehat H]\lesssim [K_2]$. So, $[K_2]$ is the greatest lower bound of $\{[H_1],[H_2]\}$.  
    \end{proof}
    Note that in $\text{AutCl}(Q_{4m})$, $[\bigl<x^d,y\bigr>]=[\bigl<x^d,x^iy\bigr>]$, for $d|m$ and $0\le i\le d-1$, and hence, without the loss of generality, in Theorem 2.10, we chose $\bigl<x^d,y\bigr>$, instead of $\bigl<x^d,x^iy\bigr>$. Since Theorem 2.10 shows that $\text{AutCl}(Q_{4m})$ is a lattice, so, it is interesting to know whether this lattice is distributive. Theorem 2.6 is essentially used to show that $\text{AutCl}(Q_{4m})$ is a distributive lattice.
    \begin{theorem}
        For positive integer $m$, the lattice $\text{AutCl}(Q_{4m})$ does not contain a sublattice isomorphic to pentagon ($N_5$).
    \end{theorem}
    \begin{proof}
        The result is true for $m=1,2$ as, $Q_{4}\cong \mathbb{Z}_4$, so, $\text{AutCl}(Q_4)\cong C_3$ and $\text{AutCl}(Q_8)\cong C_4$. For $m\ge 3$, let $m=2^{\alpha}p_1^{t_1}p_2^{t_2}\dots p_k^{t_k}$ be the prime factorization of $m$. If there exists a sublattice of AutCl($Q_{4m}$) isomorphic to $N_5$, then there are distinct elements $[H_1],[H_2],[H_3],[H_1]\vee'[H_2],[H_1]\wedge' [H_2] \in \text{AutCl}(Q_{4m})$ as depicted in Figure 5. 
\begin{figure}[h]
    \centering
    \begin{tikzpicture}[scale=1.5,  every node/.style={circle, fill=black, minimum size=6pt, inner sep=0pt}]
				
				% Nodes in set U
				\node[label=left:${[H_2]}$] (u2) at (0,1) {};
				\node[label=left:${[H_1]}$] (u1) at (0,0.5) {};
				\node[label=right:${[H_1]\vee'[H_3]}$] (u5) at (0.5,1.5) {};
                \node[label=right:${[H_3]}$] (u3) at (1,0.75) {};
				\node[label=right:${[H_1]\wedge'[H_3]}$] (u4) at (0.5,0) {};
				
				% Edges
				\draw (u1) -- (u2);
				\draw (u2) -- (u5);
                \draw (u1) -- (u4);
                \draw (u4) -- (u3);
                \draw (u3) -- (u5);

		\end{tikzpicture}
        \caption{}
\end{figure}\\
Now, consider the following cases:\vspace{0.2cm} \\
\textbf{Case 1:} If $H_1$ and $H_3$ are subgroups of $Q_{4m}$ with $H_1=\bigl<x^{2^{\beta_1}p_1^{u_1}p_2^{u_2}\dots p_k^{u_k}}\bigr>$ and $H_3=$ $\bigl<x^{2^{\beta_2}{p_1^{v_1}p_2^{v_2}\dots p_k^{v_k}}}\bigr>$, $0\le u_i,v_i\le t_i$, $1\le i\le k$ and $0\le\beta_j \le \alpha$, $j=1,2$, then $[H_1]\vee' [H_3]=$    $[H_2]\vee' [H_3]=[K]$, where $K=\bigl<x^{2^{\min\{\beta_1,\beta_2\}}p_1^{\text{min}\{u_1,v_1\}}p_2^{\text{min}\{u_2,v_2\}}\dots p_k^{\text{min}\{u_k,v_k\}}}\bigr>$ and which implies $[H_2]\lesssim [K]$. Since $K$ is a subgroup of $\bigl<x\bigr>$, by Theorem 2.9, the class $[K]$ is singleton. Also, as $[H_2]\lesssim [K]\lesssim [\bigl<x\bigr>]$, if $H_2$ contains $x^iy$, for some $i$, then by Theorem 2.9, $[K]$ also contains $x^iy$, for some $i$, which is not possible and hence $[H_2]$ is also singleton. Thus, $H_2=\bigl<x^{2^{\beta'}{ p_1^{l_1}p_2^{l_2}\dots p_k^{l_k}}}\bigr>$, where, $\text{min}\{u_i,v_i\}\le l_i\le u_i$ and $\min\{\beta_1,\beta_2\}\le \beta'\le \beta_1$. As, $[H_2]\vee' [H_3]=[H_1]\vee' [H_3]$, we have $\bigl<x^{2^{\min\{\beta',\beta_2\}}p_1^{\text{min}\{l_1,v_1\}}p_2^{\text{min}\{l_2,v_2\}}\dots p_k^{\text{min}\{l_k,v_k\}}}\bigr>=\bigl<x^{2^{\min{\{\beta_1,\beta_2\}}}p_1^{\text{min}\{u_1,v_1\}}p_2^{\text{min}\{u_2,v_2\}}\dots p_k^{\text{min}\{u_k,v_k\}}}\bigr>$. On comparing the order of generators, we get,
\begin{equation*}
\frac{2m}{\gcd(2m, 2^{\min\{\beta',\beta_2\}}p_1^{\text{min}\{l_1,v_1\}}\dots p_k^{\text{min}\{l_k,v_k\}}   )}=\frac{2m}{\gcd(2m, 2^{\min\{\beta_1,\beta_2\}}p_1^{\text{min}\{u_1,v_1\}}\dots p_k^{\text{min}\{u_k,v_k\}})}
\end{equation*}
and which implies $2^{\min\{\beta',\beta_2\}}p_1^{\text{min}\{l_1,v_1\}}\dots p_k^{\text{min}\{l_k,v_k\}}=2^{\min\{\beta_1,\beta_2\}}p_1^{\text{min}\{u_1,v_1\}}\dots p_k^{\text{min}\{u_k,v_k\}}$. Therefore, $\text{min}\{l_i,v_i\}=\text{min}\{u_i,v_i\}$, for all $i$, and $\min\{\beta',\beta_2\}=\min\{\beta_1,\beta_2\}$.
Moreover, as $[H_2]\wedge' [H_3]=[H_1]\wedge'[H_3]$, we have $\bigl<x^{2^{\max\{\beta',\beta_2\}}p_1^{\text{max}\{l_1,v_1\}}p_2^{\text{max}\{l_2,v_2\}}\dots p_k^{\text{max}\{l_k,v_k\}}}\bigr>=  $ $ \bigl<x^{2^{\max\{\beta_1,\beta_2\}}p_1^{\text{max}\{u_1,v_1\}}p_2^{\text{max}\{u_2,v_2\}}\dots p_k^{\text{max}\{u_k,v_k\}}}\bigr>$. On comparing the order of generators, we get,
\begin{equation*}
\frac{2m}{\gcd(2m, 2^{\max\{\beta',\beta_2\}}p_1^{\text{max}\{l_1,v_1\}}\dots p_k^{\text{max}\{l_k,v_k\}}   )}=\frac{2m}{\text{gcd}(2m, 2^{\max\{\beta_1,\beta_2\}}p_1^{\text{max}\{u_1,v_1\}}\dots p_k^{\text{max}\{u_k,v_k\}})}
\end{equation*}
and which implies $p_1^{\text{max}\{l_1,v_1\}}\dots p_k^{\text{max}\{l_k,v_k\}}=p_1^{\text{max}\{u_1,v_1\}}\dots p_i^{\text{max}\{u_k,v_k\}}$. Therefore, for all $i$, $\text{max}\{l_i,v_i\}=\text{max}\{u_i,v_i\}$ and $\max\{\beta',\beta_2\}=\max\{\beta_1,\beta_2\}$, which implies $\beta'=\beta_1$ and $l_i=u_i$, for all $i$, and consequently $[H_1]=[H_2]$, a contradiction. \vspace{0.2cm} \\
\textbf{Case 2:} Let $H_1$ and $H_3$ are subgroups of $Q_{4m}$ with atleast one subgroup in $[H_1]$ or $[H_3]$ contains $x^iy$, for some $i$. \vspace{0.2cm}\\
\textbf{Subcase 2.1:} If both $[H_1]$ and $[H_3]$ contain subgroups containing $x^iy$, for some $i$, then without the loss of generality, let $H_1=\bigl<x^{p_1^{u_1}p_2^{u_2}\dots p_k^{u_k}},y\bigr>$ and $H_3=\bigl<x^{p_1^{v_1}p_2^{v_2}\dots p_k^{v_k}},y\bigr>$, certainly $H_2$ also contains $x^iy$, for some $i$. Let $K_l=H_l\backslash\{x^iy\hspace{0.2cm}|\hspace{0.2cm}0\le i \le 2m-1\}$, then $K_l \le H_l,$ for $l=1,2,3$ and therefore, by the choices of $H_1,H_3$ and by Theorem 2.10, $[K_1], [K_2],[K_3],[K_1]\vee' [K_3],[K_1]\wedge' [K_3]$ are distinct in $\text{AutCl}(Q_{4m})$, as shown in Figure 6, which is not possible by case 1. \vspace{0.2cm}\\
\textbf{Subcase 2.2:} If only subgroups of $[H_1]$ contains $x^iy$, for some $i$, then without the loss of generality, $H_1=\bigl<x^{p_1^{u_1}p_2^{u_2}\dots p_k^{u_k}},y\bigr>$ and $H_3=\bigl<x^{2^\beta p_1^{v_1}\dots p_k^{v_k}}\bigr>$, therefore, $H_2$ contains $x^iy$, for some $i$. Let $K_l=H_l\backslash\{x^iy\hspace{0.2cm}|\hspace{0.2cm}0\le i \le 2m-1\}$, then $K_l \le H_l,$ for $l=1,2,3$ and by Theorem 2.10, $[K_1]\vee' [K_3]=[K_2]\vee'[K_3]=[\bigl<x^{p_1^{\min\{u_1,v_1\}}p_2^{\min\{u_2,v_2\}}\dots p_k^{\min\{u_k,v_k\}}}\bigr>]$ and $[K_1]\wedge' [K_3]=[K_2]\wedge'[K_3]=[\bigl<x^{2^{\beta}p_1^{\text{max}\{u_1,v_1\}}p_2^{\text{max}\{u_2,v_2\}}\dots p_k^{\text{max}\{u_k,v_k\}}}\bigr>]$. By a similar argument as in Case 5 of Theorem 2.7, we have distinct $[K_1],[K_2],[K_3],[K_1]\vee'[K_3],[K_1]\wedge' [K_3]\in \text{AutCl}(Q_{4m})$ as in Figure 6, which is not possible by case 1. 
\begin{figure}[h!]
    \centering
    \begin{tikzpicture}[scale=1.5,  every node/.style={circle, fill=black, minimum size=6pt, inner sep=0pt}]
				
				% Nodes in set U
				\node[label=left:${[K_2]}$] (u2) at (0,1) {};
				\node[label=left:${[K_1]}$] (u1) at (0,0.5) {};
				\node[label=right:${[K_1]\vee'[K_3]}$] (u5) at (0.5,1.5) {};
                \node[label=right:${[K_3]}$] (u3) at (1,0.75) {};
				\node[label=right:${[K_1]\wedge'[K_3]}$] (u4) at (0.5,0) {};
				
				% Edges
				\draw (u1) -- (u2);
				\draw (u2) -- (u5);
                \draw (u1) -- (u4);
                \draw (u4) -- (u3);
                \draw (u3) -- (u5);

		\end{tikzpicture}
        \caption{}
\end{figure}\\
\textbf{Subcase 2.3:} If only subgroups in $[H_3]$ contains $x^iy$, for some $i$, then without the loss of generality, let $H_1=\bigl<x^{2^{\beta}p_1^{u_1}p_2^{u_2}\dots p_k^{u_k}}\bigr>$ and $H_3=\bigl<x^{p_1^{v_1}p_2^{v_2}\dots p_k^{v_k}},y\bigr>$, then certainly $H_2$ does not contain $x^iy$, for any $i$. Let $K_l=H_l \backslash \{x^iy\hspace{0.2cm}|\hspace{0.2cm} 0\le i\le 2m-1\} $, for $l=1,2,3$, then clearly, $K_l\le H_l$ and by Theorem 2.10, $[K_1]\vee' [K_3]=[K_2]\vee'[K_3]=[\bigl<x^{p_1^{\min\{u_1,v_1\}}p_2^{\min\{u_2,v_2\}}\dots p_k^{\min\{u_k,v_k\}}}\bigr>]$ and $[K_1]\wedge' [K_3]=[K_2]\wedge'[K_3]=[\bigl<x^{2^{\beta}p_1^{\text{max}\{u_1,v_1\}}p_2^{\text{max}\{u_2,v_2\}}\dots p_k^{\text{max}\{u_k,v_k\}}}\bigr>]$. By a similar argument as in Case 4 of Theorem 2.7, we have distinct $[K_1],[K_2],[K_3],[K_1]\vee'[K_3],[K_1]\wedge' [K_3]\in \text{AutCl}(Q_{4m})$ as in Figure 6, which is not possible by case 1.
    \end{proof}
    \begin{theorem}
        For positive integer $m$, the lattice $\text{AutCl}(Q_{4m})$ does not contain a sublattice isomorphic to diamond $(M_3)$.
    \end{theorem}
    \begin{proof}
        The result is true for $m=1,2$ as, $Q_{4}\cong \mathbb{Z}_4$, so $\text{AutCl}(Q_4)\cong C_3$ and $\text{AutCl}(Q_8)\cong C_4$. For $m\ge 3$, let $m=2^{\alpha}p_1^{t_1}p_2^{t_2}\dots p_k^{t_k}$ be the prime factorization of $m$. Suppose that there exists a sublattice of AutCl($Q_{4m}$) isomorphic to $M_3$, then one can find distinct elements $[H_1],[H_2],[H_3],[H_1]\vee'[H_2],[H_1]\wedge' [H_2] \in \text{AutCl}(Q_{4m})$ as depicted in Figure 7.
\begin{figure}[h]
    \centering
    \begin{tikzpicture}[scale=1.5,  every node/.style={circle, fill=black, minimum size=6pt, inner sep=0pt}]
				
				% Nodes in set U
				\node[label=left:${[H_1]}$] (u1) at (-0.7,0) {};
				\node[label=right:${[H_3]}$] (u3) at (0.7,0) {};
				\node[label=left:${[H_1]\vee'[H_2]}$] (u4) at (0,0.7) {};
				\node[label=left:${[H_1]\wedge'[H_2]}$] (u5) at (0,-0.7) {};
                \node[label=left:${[H_2]}$] (u2) at (0,0) {};
            
				% Edges
				\draw (u1) -- (u4);
				\draw (u4) -- (u3);
				\draw (u3) -- (u5);
				\draw (u5) -- (u1);
                \draw (u5) -- (u2);
                \draw (u2) -- (u4);

		\end{tikzpicture}
        \caption{}
\end{figure}\\
Now, consider the following cases: \vspace{0.2cm} \\
\textbf{Case 1:} If both $H_1$ and $H_2$ are such that $[H_1]=[\bigl<x^{2^{\beta_1}p_1^{u_1}p_2^{u_2}\dots p_k^{u_k}}\bigr>]$ and $[H_2]=$ $[\bigl<x^{2^{\beta_2}p_1^{v_1}p_2^{v_2}\dots p_k^{v_k}}\bigr>]$, $0\le u_i,v_i\le t_i$, $1\le i\le k$ and $0\le\beta_j \le\alpha$, $j=1,2$, so, $[H_1]\vee' [H_2]=[H_1]\vee' [H_3]=[H_2]\vee'[H_3]=[K]$, where $K=\bigl<x^{2^{\min\{\beta_1,\beta_2\}}p_1^{\min\{u_1,v_1\}}p_2^{\min\{u_2,v_2\}}\dots p_k^{\min\{u_k,v_k\}}}\bigr> $. Therefore, $H_3$ is a subgroup of $\bigl<x\bigr>$. Let $H_3=\bigl<x^{2^{\beta'}p_1^{l_1}\dots p_k^{l_k}}\bigr>$, as, each class is singleton and $[H_1]\vee' [H_2]=[H_1]\vee' [H_3]$, we have $\bigl<x^{2^{\min\{\beta_1,\beta_2\}}p_1^{\min\{u_1,v_1\}}p_2^{\min\{u_2,v_2\}}\dots p_k^{\min\{u_k,v_k\}}}\bigr>$ = $\bigl<r^{2^{\min\{\beta_1,\beta'\}}p_1^{\min\{u_1,l_1\}}p_2^{\min\{u_2,l_2\}}\dots p_k^{\min\{u_k,l_k\}}}\bigr>$. On comparing the order of their generators, we get, 
\begin{equation*}
 \frac{2m}{\gcd(2m, 2^{\min\{\beta_1,\beta_2\}}p_1^{\min\{u_1,v_1\}}\dots p_k^{\min\{u_k,v_k\}})}=   \frac{2m}{\gcd(2m, 2^{\min\{\beta_1,\beta'\}}p_1^{\min\{u_1,l_1\}}\dots p_k^{\min\{u_k,l_k\}}   )}
\end{equation*}
and which implies $2^{\min\{\beta_1,\beta_2\}}p_1^{\min\{u_1,v_1\}}\dots p_i^{\min\{u_k,v_k\}}=2^{\min\{\beta_1,\beta'\}}p_1^{\min\{u_1,l_1\}}\dots p_k^{\min\{u_k,l_k\}}$. Therefore, $\min\{u_i,v_i\}=\min\{u_i,l_i\}$, for all $i$, and $\min\{\beta_1,\beta_2\}=\min\{\beta_1,\beta'\}$. Furthermore, as $[H_1]\wedge' [H_2]=[H_1]\wedge'[H_3]$, we have $\bigl<r^{2^{\max\{\beta_1,\beta_2\}}p_1^{\text{max}\{u_1,v_1\}}p_2^{\text{max}\{u_2,v_2\}}\dots p_k^{\text{max}\{u_k,v_k\}}}\bigr>=\\\bigl<r^{2^{\max{\{\beta_1,\beta'\}}}p_1^{\text{max}\{u_1,l_1\}}p_2^{\text{max}\{u_2,l_2\}}\dots p_k^{\text{max}\{u_k,l_k\}}}\bigr>$. On comparing the order of their generators, we get,
\begin{equation*}
\frac{2m}{\gcd(2m, 2^{\max\{\beta_1,\beta_2\}}p_1^{\text{max}\{u_1,v_1\}}\dots p_k^{\text{max}\{u_k,v_k\}})}=\frac{2m}{\gcd(2m, 2^{\max\{\beta_1,\beta'\}}p_1^{\text{max}\{u_1,l_1\}}\dots p_k^{\text{max}\{u_k,l_k\}})}
\end{equation*}
and which implies $2^{\max\{\beta_1,\beta_2\}}p_1^{\text{max}\{u_1,v_1\}}\dots p_k^{\text{max}\{u_k,v_k\}}=2^{\max\{\beta_1,\beta'\}}p_1^{\text{max}\{u_1,l_1\}}\dots p_i^{\text{max}\{u_k,l_k\}}$. Therefore, for all $i$, $\text{max}\{u_i,v_i\}=\text{max}\{u_i,l_i\}$ and $\max\{\beta_1,\beta_2\}=\max\{\beta_1,\beta'\}$ which implies $\beta'=\beta_2$ and $l_i=v_i$, for all $i$, and consequently, $[H_2]=[H_3]$, a contradiction.\vspace{0.2cm}\\
\textbf{Case 2:} Let $H_1$ and $H_2$ are subgroups of $Q_{4m}$ with atleast one subgroup in $[H_1]$ or $[H_2]$ contains $x^iy$, for some $i$. \vspace{0.2cm}\\
\textbf{Subcase 2.1:} If both $[H_1]$ and $[H_2]$ contains subgroups containing $x^iy$, for some $i$, then subgroups in the class $[H_3]$ also contains $x^iy$, for some $i$. Now, let $K_l=H_l\backslash\{x^iy\hspace{0.2cm}|\hspace{0.2cm}0\le i\le 2m-1\}$, for $l=1,2,3$, so that, $K_l\le H_l$ and by the choices of $H_1,H_2$ and by Theorem 2.10, $[K_1], [K_2],[K_3],[K_1]\vee' [K_2],[K_1]\wedge' [K_2]$ are distinct in $\text{AutCl}(Q_{4m})$, as shown in Figure 8, which is not possible by case 1. 
\begin{figure}[h]
    \centering
    \begin{tikzpicture}[scale=1.5,  every node/.style={circle, fill=black, minimum size=6pt, inner sep=0pt}]
				
				% Nodes in set U
				\node[label=left:${[K_1]}$] (u1) at (-0.7,0) {};
				\node[label=right:${[K_3]}$] (u3) at (0.7,0) {};
				\node[label=left:${[K_1]\vee'[K_2]}$] (u4) at (0,0.7) {};
				\node[label=left:${[K_1]\wedge'[K_2]}$] (u5) at (0,-0.7) {};
                \node[label=left:${[K_2]}$] (u2) at (0,0) {};
            
				% Edges
				\draw (u1) -- (u4);
				\draw (u4) -- (u3);
				\draw (u3) -- (u5);
				\draw (u5) -- (u1);
                \draw (u5) -- (u2);
                \draw (u2) -- (u4);

		\end{tikzpicture}
        \caption{}
\end{figure}\\
\textbf{Subcase 2.2:} If $[H_1]$ contains a subgroup containing $x^iy$, for some $i$, but not $[H_2]$, then as $[H_1]\vee'[H_2]=[H_2]\vee'[H_3]$, so, $[H_3]$ contain subgroups containing $x^iy$, for some $i$, and hence, by Theorem 2.10, subgroups in the class $[H_1]\wedge'[H_3]$ also contains $x^iy$, for some $i$, but this is not possible as, by Theorem 2.10, no subgroups in $[H_2]\wedge'[H_3]$ contains $x^iy$, for any $i$.
    \end{proof}
\section{Finite Groups whose Automorphic Classes are Chain}
In order to characterize $\text{AutCl}(G)$ to be a chain, we essentially need the following results. 
\begin{theorem} \cite{suzuki1986group} The following three conditions on a $p$-group are equivalent.
\begin{enumerate}
    \item Every abelian subgroup is cyclic.
    \item There is exactly one subgroup of order $ p$.
    \item The group $ G$ is either cyclic or a generalized quaternion group $Q_{2^n}, n\ge3.$.
\end{enumerate} \end{theorem}
\begin{theorem} \cite{suzuki1982group} Let $A$ be an abelian normal subgroup of maximal order of a $p$-group $G$. If $|G| = p^n$ and $|A| = p^a$, we have $2n \le a(a + 1)$. \end{theorem}
\begin{theorem} Let $G$ be a finite group. The poset $\text{AutCl}(G)$ is a chain if and only if $G$ is one of the following: \begin{enumerate}
    \item A cyclic $p$-group,
    \item An elementary abelian $p$-group, 
    \item Quaternion group of order 8.
\end{enumerate} 
\end{theorem}
\begin{proof} Suppose that $\text{AutCl}(G)$ is a chain. Then $G$ must be a $p$-group else there exist distinct prime factors $p_1$ and $p_2$ of $|G|$. By Sylow's first theorem there exist subgroups $H_1$ and $H_2$ of $G$ of order $p_1$ and $p_2$, respectively. Therefore, $[H_1]$ and $[H_2]$ are not comparable in $\text{AutCl}(G)$, which is a contradiction. Therefore, $|G|=p^n$, for some $n$, and choose a minimal normal subgroup $H$ of $G$. It is known that $p$-group of order $p^n$ has a normal subgroup of order $p^k$ for each $k$, where $0\le k\le n$ \cite{Dummit_Foote_}. Therefore, as $H$ is minimal normal, the order of $H$ is $p$. \vspace{0.2cm} \\
\textbf{Case 1:} If $H$ is the unique subgroup of order $p$ of $G$, then by Theorem 3.1, $G$ is either a cyclic or a generalized quaternion group $Q_{2^n}, n\ge 3.$ For $n\ge 4$, we have $[Q_{2^{n-1}}]$ and $[\mathbb{Z}_{2^{n-1}}]$ are distinct coatoms of $\text{AutCl}(Q_{2^n})$. Therefore, for $n\ge 4$, $\text{AutCl}(Q_{2^n})$ is not a chain and in this case the only possible group $G$ with $\text{AutCl}(G)$ being a chain is cyclic $p$-group or the quaternion group $Q_8$. \vspace{0.2cm}\\
\textbf{Case 2:} If $H$ is not the unique subgroup of $G$ of order $p$, then $G$ has a minimal subgroup $K$ with $K\ne H$. As, $|HK|=\frac{|H \times K|}{|H\wedge K|}$, we have $HK$ is a subgroup of order $p^2$. Note that $HK\cong H\times K$. So $HK$ is an elementary abelian group of order $p^2$. Now, if $G$ contains a cyclic group of order $p^2$, say $H_1$ then the class $[H_1]$ is distinct and incomparable from $[HK]$ as $H_1$ is cyclic of order $p^2$ and $HK$ is an elementary abelian group of order $p^2$, this would contradict the fact that $\text{AutCl}(G)$ is a chain. Consequently, $\text{exp}(G)=p$. \vspace{0.2cm}\\
\textbf{Subcase 2.1:} If $G$ is abelian, then clearly it is an elementary abelian $p$-group. \vspace{0.2cm}\\
\textbf{Subcase 2.2:} If $G$ is non abelian, then $G$ contains a non abelian subgroup of order $p^3$, say $N$. Let $A$ be an abelian normal subgroup of maximal order of $G$ with $|A|=p^a$. If $a\ge 3$ then $A$ has a subgroup $A_1$ of order $p^3$. Note that $A_1$ is abelian and therefore $[A_1]$ and $[N]$ are incomparable, a contradiction. Consequently, $a\in \{1,2\}$. By Theorem 3.2, we have that $2n\le a(a+1)$ and which implies $n\le 3$. For $n=1,2$, we have $G$ is abelian and therefore $n=3$. Hence, $G$ is a non abelian $p$-group of order $p^3$ and exponent $p$.\par
Conversely, if $G$ is a cyclic $p$-group then $G\cong \mathbb{Z}_{p^n}$, for some $n$, by example 1, $\text{AutCl}(\mathbb{Z}_{p^n})\cong L(\mathbb{Z}_{p^n})$. As $L(\mathbb{Z}_{p^n})$ is a chain, so is $\text{AutCl}(\mathbb{Z}_{p^n})$. Also, it is clear that $\text{AutCl}(Q_8)$ is a chain. \par Now, let $G=\underbrace{\mathbb{Z}_p\times...\times \mathbb{Z}_p}_\textrm{\emph{n} copies}$ be the elementary abelian $p$ group of order $p^n$. Let $H$ be a subgroup of $G$. Since $G$ is elementary abelian $p$-group, so is $H$. Moreover, $H$ is a subspace of $G$ over $\mathbb{Z}_p$, so choose a basis $B'$=$\{\zeta_1,\zeta_2\dots \zeta_k\}$ of $H$ over $\mathbb{Z}_p$ and note that $H=\bigl<\zeta_{1},\zeta_{2},\dots,\zeta_{k}\bigr>$. Now, extend the set $B'$ to a basis $B=\{\zeta_1,\zeta_2,\dots \zeta_k, \zeta_{k+1},\zeta_{k+2},\dots,\zeta_n\}$ of $\mathbb{Z}_p\times\dots \times \mathbb{Z}_p$ over $\mathbb{Z}_p$. Then the map $f:G \to G$ with $f(e_i)=\zeta_i$ for $1\le i\le n$ is an automorphism of $G$, where $\{e_1,e_2,\dots,e_n\}$ is the standard basis of $\mathbb{Z}_p\times \dots \times \mathbb{Z}_p$ over $\mathbb{Z}_p$ and $f(\underbrace{\mathbb{Z}_p\times\dots \mathbb{Z}_p}_\textrm{\emph{k} copies}\times\{0\}\times \dots\times\{0\})=H$. Therefore, for any divisor $p^k$ of $p^n$, there exists exactly one class of subgroup of order $p^k$ for $0\le k\le n$. Hence, $\text{AutCl}(G)$ is a chain. \par
Now, let $G$ be a non abelian group of order $p^3$ with exponent $p$. Then $G$ is isomorphic to the Heisenberg group $\text{Heis}(\mathbb{Z}_3)$. So, $G\cong \bigl<x,y,z \hspace{0.2cm}|\hspace{0.2cm} xyx^{-1}y^{-1}=z, \hspace{0.2cm}xz=zx, \hspace{0.2cm} yz=zy, \hspace{0.2cm}x^p=y^p=z^p=e \bigr>$. We will show that $[\bigl<x\bigr>]$ and $[\bigl<z\bigr>]$ are incomparable in the poset $\text{AutCl}(G)$, i.e; there exists no $f\in \text{Aut}(G)$ with $f(\bigl<x\bigr>)=\bigl<z\bigr>$. Since $x^{-1}zx=x^{-1}xz=z\in \bigl<z\bigr>$ and $y^{-1}zy=y^{-1}yz=z\in \bigl<z\bigr>$, we have $\bigl<z\bigr>$ is normal in $G$. To show that $[\bigl<x\bigr>]$ and $[\bigl<z\bigr>]$ are incomparable, it is sufficient to show that $\bigl<x\bigr>$ is not normal in $G$. For if $\bigl<x\bigr>$ is normal in $G$, then $y^{-1}xy\in \bigl<x\bigr>$.\vspace{0.2cm}\\
\textbf{Case 1:} If $y^{-1}xy=e$, then $xy=y$ and which implies $x=e$, a contradiction.\vspace{0.2cm}\\ 
\textbf{Case 2:} If $y^{-1}xy=x$ then $xy=yx$. But $xz=zx$ and $yz=zy$ which together implies $G$ is an abelian group, a contradiction. \vspace{0.2cm}\\
\textbf{Case 3:} Lastly, if $y^{-1}xy=x^k$, $2\le k \le p-1$, then, $xy=yx^k$. As, $xyx^{-1}y^{-1}=z$ implies $yx^{k-1}y^{-1}=z$ and this implies $yx^{k-1}=zy=yz$ and which implies $z=x^{k-1}$. Moreover, $yz=zy$ implies that $yx^{k-1}=x^{k-1}y$ and therefore, $y^{-1}x^{k-1}y=x^{k-1}$ i.e., $(y^{-1}xy)^{k-1}=x^{k-1}$. Hence, $x^{k(k-1)}=x^{k-1}$ and this implies $x^{(k-1)^2}=e$. Since, $|x|=p$, so $p$ must divide $(k-1)^2$. As $p$ is a prime, $p$ must divide $k-1$, but $2\le k \le p-1$, a contradiction. Therefore, $\bigl<x\bigr>$ is not normal in $G$. Consequently, $\text{AutCl}(G)$ is not a chain.
\end{proof}
\section{Conclusions and Open Problems}
In this paper, we have shown that $\text{AutCl}(D_n)$ and $\text{AutCl}(Q_{4m})$ form distributive lattices and have characterized all classes of finite groups $G$ for which $\text{AutCl}(G)$ is a chain. Following are some open problems about the poset of automorphic classes of subgroups:
\begin{enumerate}
    \item Determine classes of finite groups $G$ for which $\text{AutCl}(G)$ is a lattice. In particular, determine classes of groups $G$ for which $\text{AutCl}(G)$ is a distributive lattice (or modular lattice).
    \item Let $G_1$ and $G_2$ be two finite groups with $\text{AutCl}(G_1)\cong \text{AutCl}(G_2)$, then what can be said about the groups $G_1$ and $G_2$?
    \item A \emph{projectivity} of two groups is a lattice isomorphism of their subgroup lattices and an \emph{autoprojectivity} of a group is a projectivity from group to itself, thus one can generalize the poset $\text{AutCl}(G)$ associated to a finite group $G$ by considering autoprojectivities instead of group automorphisms. The set of all autoprojectivities of $G$ is denoted by $P(G)$ \cite{Schmidt+1994}. Consider the following set: \begin{equation*}
        \text{ AutCl}'(G)= \{[H]'\hspace{0.2cm} |\hspace{0.2cm} H\in L(G)\},\hspace{0.2cm} \text{where} \hspace{0.2cm}[H]'=\{K\in L(G)\hspace{0.2cm}|\hspace{0.2cm} \text{there is}\hspace{0.1cm} f\in P(G) \hspace{0.1cm}\text{with}\hspace{0.1cm} f(H)=K\}.
    \end{equation*}
    Investigate the above set with respect to an analogous ordering relation as that of $\text{AutCl}(G)$.
    \item It is interesting to know the structure of the poset $\text{AutCl}(G)$, when $G=\text{Heis}(\mathbb{Z}_p)$. We speculate that the poset is isomorphic to Figure 3. We have verified this for several small groups using GAP \cite{GAP4}. However, we failed in proving the following: \vspace{0.2cm}\\
\textbf{Conjecture:} For any odd prime $p$, the poset $\text{AutCl}(\text{Heis}(\mathbb{Z}_p))$ is isomorphic to the poset as shown in Figure 9.
\end{enumerate}
\begin{figure}[h]
    \centering
    \begin{tikzpicture}[scale=1.5,  every node/.style={circle, fill=black, minimum size=6pt, inner sep=0pt}]
				
				% Nodes in set U
				\node[label=left:] (u1) at (0,0) {};
				\node[label=left:] (u2) at (-0.5,0.5) {};
				\node[label=left:] (u3) at (0.5,0.5) {};
                \node[label=left:] (u4) at (0,1) {};
				\node[label=left:] (u5) at (0,1.5) {};
				
				% Edges
				\draw (u1) -- (u2);
				\draw (u2) -- (u4);
                \draw (u4) -- (u5);
                \draw (u1) -- (u3);
                \draw (u3) -- (u4);

		\end{tikzpicture}
        \caption{}
\end{figure}
\section*{Acknowledgement}
The second author is thankful to the Council of Scientific and Industrial Research (CSIR), for financial assistance in the form of Junior Research Fellowship (JRF), bearing the File Number: 09/0414(22038)/2025-EMR-I.

\printbibliography
\end{document}